\documentclass[10pt,a4paper]{scrartcl}

\usepackage[left=3cm,right=4cm,top=3cm,bottom=5cm,includeheadfoot]{geometry}
\usepackage[english]{babel}
\usepackage[utf8]{inputenc}
\usepackage{inputenc, amsmath, amssymb, latexsym, graphicx, fancyhdr, amsthm}
\usepackage{mathtools}
\usepackage[mathscr]{euscript}
\usepackage{enumitem}

\renewcommand{\:}{\colon}

\newtheorem{theorem}{Theorem}[section]
\newtheorem{corollary}[theorem]{Corollary}
\newtheorem{proposition}[theorem]{Proposition}
\newtheorem{lemma}[theorem]{Lemma}

\theoremstyle{definition}

\newtheorem{example}[theorem]{Example}
\newtheorem*{example*}{Example}
\newtheorem{remark}[theorem]{Remark}

\def\mathclap#1{\text{\hbox to 0pt{\hss$\mathsurround=0pt#1$\hss}}}

\newcommand{\T}{\ensuremath{\mathbb{T}}}
\renewcommand{\S}{\ensuremath{\mathbb{S}}}

\newcommand{\R}{\ensuremath{\mathbb{R}}}
\newcommand{\Z}{\ensuremath{\mathbb{Z}}}
\newcommand{\N}{\ensuremath{\mathbb{N}}}

\newcommand{\eps}{\varepsilon}
\newcommand{\abs}[1]{\left|#1\right|}
\newcommand{\Abs}[1]{\left\|#1\right\|}

\newcommand{\supp}{\ensuremath{\mathrm{supp}}}
\newcommand{\Fol}{\ensuremath{\mathcal F}}
\newcommand{\mc}{\mathcal}
\renewcommand{\phi}{\varphi}
\newcommand{\ssq}{\subseteq}
\renewcommand{\=}{\coloneqq}
\newcommand{\mef}{\T}
\newcommand{\tfg}[1]{[\![#1]\!]}

\newcommand{\Per}{\textrm{Per}}

\DeclareOldFontCommand{\bf}{\normalfont\bfseries}{\mathbf}
\DeclareOldFontCommand{\tt}{\normalfont\ttfamily}{\mathtt}
\DeclareOldFontCommand{\rm}{\normalfont\rmfamily}{\mathrm}

\makeatletter
\let\@fnsymbol\@arabic
\makeatother

\title{\bf\Large The structure of mean equicontinuous group actions}

\author{
    \normalsize Gabriel Fuhrmann\thanks{Department of Mathematics, Imperial
        College London, UK. Email: {\tt gabriel.fuhrmann@imperial.ac.uk}}
    \and\normalsize Maik Gröger\thanks{Faculty of Mathematics, University of Vienna,
        Austria. Email: {\tt maik.groeger@univie.ac.at}}
    \and\normalsize Daniel Lenz\thanks{Department of Mathematics, University of Jena,
        Germany. Email: {\tt daniel.lenz@uni-jena.de}}}
\date{\vspace{-0.5cm}}

\begin{document}

\maketitle
\begin{abstract}
    We study mean equicontinuous actions of locally compact
    $\sigma$-compact amenable groups on compact metric spaces.
    In this setting, we establish the equivalence of mean equicontinuity
    and topo-isomorphy to the maximal equicontinuous factor
    and provide a characterization of mean equicontinuity of an action
    via properties of its product.
    This characterization enables us to show the equivalence of mean equicontinuity
    and the weaker notion of Besicovitch-mean equicontinuity in fairly high generality,
    including actions of abelian groups as well as minimal actions of general groups.
    In the minimal case, we further conclude that mean equicontinuity is equivalent
    to discrete spectrum with continuous eigenfunctions.
    Applications of our results yield a new class of non-abelian mean equicontinuous
    examples as well as a characterization of those extensions of mean
    equicontinuous actions which are still mean equicontinuous.
\end{abstract}

\section{Introduction}\label{sec:introduction}

Isometric actions on compact metric spaces constitute fundamental objects of study
in the field of dynamical systems.
In fact, despite possessing structurally simple dynamics, they relate to deep
problems of general mathematical interest.
Already rigid rotations on the circle have close connections to continued fraction
expansions (see, for example, \cite{Series1985}), the rich theory of discrepancy
of sequences (see, e.g., \cite{DrmotaTichy1997} and references therein), or the
Three Distance Problem and its versatile generalizations
(see, for instance, \cite{AlessandriBerthe1998}), to name but a few.
With their dynamical simplicity on the one hand and the
relevance of such problems on the other hand, it is natural to take
actions by isometries as a point of departure in the endeavor to understand
topological dynamical systems in general.

Actually, a substantial part of the abstract theory
of topological dynamics can be understood as dealing with the following issue:
given a general action which is not isometric, how close is this action to an isometric one?
An essential tool in answering this question is the so-called maximal
equicontinuous factor (or, topologically equivalent, the maximal isometric factor)
of a given action.
Now, with this canonical factor at hand, we may restate the above question in the
following way: what is the regularity of the corresponding factor map?

Of course, various regularity features can be (and have been) considered.
On the topological side, it is natural to investigate the existence of points
where this factor map is one-to-one and this leads to the notion of almost
automorphic actions \cite{Veech1965}.
Once an invariant measure $\mu$ is given, one can also ask for injectivity
of the factor map for almost all points  with respect to $\mu$ and in many contexts
this is referred to as regularity of the system.

With a more measure-theoretical flavor, we may study factor maps that establish a measure
isomorphy with respect to all invariant measures and their push-forward on the maximal equicontinuous factor.
This is the starting point of the current article and we present a
comprehensive treatment of actions which allow for such factor maps.
Our first main result gives a characterization of these actions in terms of a
weakening of isometry known
as mean equicontinuity (Theorem \ref{thm:structural result mean
equicontinuous systems}).
Our subsequent results then unfold the
notion of mean equicontinuity in terms of product systems (Theorem
\ref{thm:intro equivalence criterion of mean equi. involving
pointwise u.e}) and provide a spectral characterization of mean
equicontinuity (Theorem \ref{thm:spectral-characterization}) for minimal actions.
A priori, the concept of mean equicontinuity comes in two variants, one
known as Weyl-mean equicontinuity and the other as Besicovitch-mean equicontinuity.
Along the way, we derive sufficient conditions for these two notions
to agree (Theorem \ref{thm:F-mean-equicontinuity implies mean-equicontinuity}).

The concepts of Weyl- and Besicovitch-mean equicontinuity were introduced in \cite{LiTuYe2015}
for integer actions.
In fact, in this case the notion of Besicovitch-mean equicontinuity is immediately seen to be
equivalent to the concept of mean Lyapunov-stability which was already introduced
in 1951 by Fomin \cite{Fomin1951} in the context of $\Z$-actions with discrete spectrum.
Later, a first systematic treatment was carried out by Auslander \cite{Auslander1959}.

Our results tie in with various recent streaks of investigations:
for $\Z$-actions, there is the fundamental work of Downarowicz
and Glasner on mean equicontinuity \cite{DownarowiczGlasner2016}, providing a
detailed study in the minimal case.
Our results generalize these results from the group of integers to general locally
compact $\sigma$-compact amenable groups.
In the main structural characterization given in Theorem \ref{thm:structural result mean
equicontinuous systems}, we can also completely remove the minimality condition.
Further, in our treatment of the relation between Weyl- and Besicovitch-mean equicontinuity, we
can remove the minimality condition in many cases as well and thereby generalize
\cite{QiuZhao} which treats the case of general (that is, not necessarily minimal) $\Z$-actions.

Concerning abelian groups, mean equicontinuity and its relation to the spectral theory
of dynamical systems (in particular, to discrete spectrum) has been studied by
various groups \cite{Garcia-Ramos2017,Lenz,Garcia-RamosMarcus2019}.
Indeed, these works feature weaker versions of mean equicontinuity in order to
characterize discrete spectrum.
So, the restriction of our spectral result to the abelian case, given in
Corollary \ref{cor:spectral-characterization-abelian-case},  can be
seen as a natural complement to these works. More specifically, our
spectral characterization shows --in the minimal case--  that mean
equicontinuity is equivalent to unique ergodicity and discrete
spectrum together with the continuity of eigenfunctions (see also
\cite{DownarowiczGlasner2016} for the case of $\Z$-actions).

Discrete spectrum is particularly relevant in the context of
aperiodic order. This field has attracted substantial attention in
the last decades due to the discovery of substances --later called
quasicrystals-- featuring this type of order (see the recent survey
collection \cite{KellendonkLenzSavinien2015} and the monograph
\cite{BaakeGrimm2013} for background and further details). A basic
quantity in the study of aperiodic order is the diffraction measure
of an aperiodic configuration  and a key task is  to understand when
the diffraction measure is a pure point measure. Due to a collective
effort over the last twenty years, this turns out to be equivalent
to discrete spectrum of an associated dynamical systems, see for
instance \cite{BaakeLenz2017} for a recent survey.

There is no axiomatic framework for aperiodic order (yet). However, typical
systems studied in the context of aperiodic order have further
regularity properties such as minimality and unique ergodicity. As
discussed below (see Remark \ref{rem:mean equicontinuity and
diffraction measure}), one may argue that our spectral
characterization shows that mean equicontinuous systems are the
``right" systems to model minimal systems with aperiodic order.



\subsection{Basic notation and definitions}\label{sec:Notation}

We call a triple $(X,G,\alpha)$ a \emph{(topological) dynamical system} if
$X$ is a compact metric space (endowed with a metric $d$), $G$ is a topological group and
$\alpha$ is a continuous action of $G$ on $X$ by homeomorphisms.
Here, continuity of $\alpha$ is understood as continuity of the map $G\times X \ni(g,x)\mapsto \alpha(g)(x)\in X$.
Most of the time, we will keep the action $\alpha$ implicit and simply refer to $(X,G)$ as a dynamical system.
In a similar fashion, we mostly write $gx$ instead of $\alpha(g)(x)$ ($g\in G,x\in X$).
For $\Z$-actions, which are uniquely determined by $f\=\alpha(1)$, we also refer by $(X,f)$
to the dynamical system $(X,\Z,\alpha)$.

A dynamical system $(Y,G)$ is a \emph{(topological) factor} of another dynamical
system $(X,G)$ if there exists a continuous surjection $h : X\longrightarrow Y$
with $h(gx)=g h(x)$ for all $g\in G$ and $x\in X$.
In this case, $h$ is called a \emph{factor map} and $(X,G)$ is referred to as
a \emph{(topological) extension} of $(Y,G)$.
If $h$ is further injective (and hence, a homeomorphism), we say $(X,G)$ and $(Y,G)$
are \emph{conjugate} and call $h$ a \emph{conjugacy}.

We say $(X,G)$ is \emph{transitive} if there is $x\in X$ whose
\emph{orbit} $Gx$ is dense.
In this case, also the point $x$ is called \emph{transitive}.
We say $(X,G)$ is \emph{minimal} if every $x\in X$ is transitive.
A set $A\subseteq X$ is \emph{invariant under $G$} (or
\emph{$G$-invariant}) if $gA=A$ for all $g\in G$.
We call a non-empty, closed and $G$-invariant set $A\subseteq X$ \emph{transitive}
if $(A,G)$ is transitive; we call such a set \emph{minimal} if $(A,G)$ is minimal.
It is a well-known consequence of Zorn's Lemma that every dynamical system $(X,G)$ has a minimal set $A\ssq X$.
Clearly, distinct minimal sets are disjoint.
Observe that invariance, transitivity and minimality are preserved under factor maps.

A system $(X,G)$ is \emph{equicontinuous} if
for every $\eps>0$ there exists $\delta_{\eps}>0$ such that
$d(x,y)<\delta_{\eps}$ implies $d(gx,gy)<\eps$ for all $g\in G$.
If $\delta_\eps$ can be chosen to equal $\eps$, then $(X,G)$ is called \emph{isometric}.
Observe that if $(X,G)$ is equicontinuous, then $\hat d(x,y)=\sup_{g\in G}d(gx,gy)$ defines a metric
that induces the same topology on $X$ as $d$.
Clearly, $\hat d(gx,gy)=\hat d(x,y)$ for all $g\in G$ and $x,y\in X$ which implies
that we can use the terms equicontinuous system and isometric system synonymously.

It is well known that  every topological dynamical system $(X,G)$ has a unique
(up to conjugacy) \emph{maximal equicontinuous factor (MEF)}, denoted by $(\mef,G)$.
That is, $(\mef,G)$ is an
equicontinuous factor of $(X,G)$ and moreover an extension of every other
equicontinuous factor of $(X,G)$.
For our considerations, the  following simple consequence of the defining
property of $(\mef,G)$ will be sufficient:  If $(\mef',G)$ is an
equicontinuous factor of $(X,G)$ with factor map $\pi'$ such that
$\pi'(x) = \pi' (y)$  implies  $\inf_{g\in G} d(gx,gy) = 0$, then
$(\mef',G)$ is conjugate to $(\mef,G)$.
For a detailed discussion of the above facts, we refer
to \cite{Auslander1988}.

Throughout this work, we consider $G$ to be a locally compact $\sigma$-compact
amenable group.
Recall that a locally compact $\sigma$-compact group $G$ is called amenable if
there exists a sequence $(F_n)_{n\in\N}$, called a \emph{(left) Følner
sequence}, of non-empty compact sets in $G$ such that
\begin{align*}
    \lim_{n\to\infty}\frac{m(gF_n\triangle F_n)}{m(F_n)}
        =0\quad\textnormal{for all }g\in G,
\end{align*}
where $\triangle$ denotes the symmetric difference and $m$ is a
\emph{(left) Haar measure} of $G$ (we may synonymously write
$\abs{F}$ for the Haar measure $m(F)$ of a measurable set $F\ssq G$).

Since $G$ acts by homeomorphisms, for each
$g\in G$ the map $g\:X\ni x\mapsto gx$ is Borel bi-measurable.
We call a Borel probability measure $\mu$ on $X$ \emph{invariant under
$G$} (or \emph{$G$-invariant}) if $\mu(A)=\mu(gA)$ for every Borel
measurable subset $A\subseteq X$ and $g\in G$.
We say a $G$-invariant measure $\mu$ is \emph{ergodic} if all Borel sets $A$ with
$\mu(A\,\triangle\, gA)=0$ ($g\in G$) verify $\mu(A)=0$ or $\mu(A)=1$.
It is well known that the amenability of $G$ ensures the existence of a $G$-invariant measure
for $(X,G)$.
Further, the set of invariant measures is convex and an invariant measure is ergodic if and only if
it is an extremal point of the set of invariant measures.
In particular, if $(X,G)$ has a unique invariant measure, this measure is necessarily ergodic and $(X,G)$ is referred to
as \emph{uniquely ergodic}.
Finally, we call a closed invariant set $A\ssq X$ \emph{uniquely ergodic} if
$(A,G)$ is uniquely ergodic.
For further information of measure-theoretic properties of dynamical systems,
see also \cite{EinsiedlerWard2011}.


\subsection{Main results}

Given a dynamical system $(X,G)$ and a Følner sequence $\Fol=(F_n)_{n\in\N}$, we
call $(X,G)$ \emph{Besicovitch-$\Fol$-mean equicontinuous} or just
\emph{$\Fol$-mean equicontinuous} if for all $\eps>0$ there exists
$\delta_\eps>0$ such that
\begin{align}\label{def:Besicovitch-mean equicontinuous}
    D_{\Fol}(x,y)
    \=\varlimsup_{n\to\infty}\frac{1}{\abs{F_n}}\int\limits_{F_n} d(tx,ty)dm(t)<\eps,
\end{align}
for all $x,y\in X$ with $d(x,y)<\delta_\eps$.
The dependence on the
Følner sequence immediately motivates the next definition which
will also be the integral notion in this article.
We say $(X,G)$ is \emph{Weyl-mean equicontinuous} or just
\emph{mean equicontinuous}  if for all $\eps>0$ there is
$\delta_\eps>0$ such that for all $x,y\in X$ with
$d(x,y)<\delta_\eps$ we have
\[
    D(x,y)\=\sup\{D_{\Fol}(x,y)\;|\;\Fol\textnormal{ is a Følner sequence}\}<\eps.
\]

Before we can proceed, a few comments are in order.
First, note that $D_\Fol$ and $D$ are pseudometrics.
Moreover, as is not hard to see, $D$ is \emph{$G$-invariant}, that is, $D(gx,gy) = D(x,y)$
for all $x,y\in X$ and $g\in G$ (for the convenience of the reader, we provide a
proof of this fact, see Proposition~\ref{prop:D-is-invariant}).
Indeed, if $G$ is abelian, it is immediately seen that \eqref{def:Besicovitch-mean equicontinuous}
already defines a $G$-invariant pseudometric (simply for algebraic reasons)
which simplifies many proofs for abelian $G$.
In the non-abelian situation, this does not hold anymore in general.
Yet, it turns out that under fairly general assumptions on $(X,G)$ it actually is true
if $(X,G)$ is mean equicontinuous (see Theorem~\ref{thm:F-mean-equicontinuity implies mean-equicontinuity}).
It is an interesting observation that in this case, however, the reason behind the invariance
of $D_\Fol$ is not so much algebraic but ergodic in nature (see Section~\ref{sec:relation Besicovitch and Weyl}).

In the following main structural result, we will see that mean
equicontinuity of a system $(X,G)$ is intimately linked to a
regularity property of the topological factor map $\pi:X\to\mef$
onto its maximal equicontinuous factor $(\mef,G)$.
For the definition of this regularity property, we need to introduce
the following notion.
Two probability spaces $(X,\mc B_X,\mu)$ and $(Y,\mc B_Y,\nu)$ are called
\emph{isomorphic} ($\!\bmod\, 0$) if there are measurable sets $M\subseteq X$ and
$N\subseteq Y$ with $\mu(M)=\nu(N)=1$ and a bi-measurable bijection $h'\: M \to N$
which is measure preserving, that is, $\mu({h'}^{-1}(A))=\nu(A)$ for all measurable
$A\ssq N$.
In this case, we call $h'$ an \emph{isomorphism} ($\bmod\, 0$) with respect to
$\mu$ and $\nu$.
We also refer to an everywhere defined measurable map $h\: X\to Y$ as an
\emph{isomorphism} ($\bmod\, 0$) with respect to $\mu$ and $\nu$ if $h(x)=h'(x)$
with $x\in M$ for some $h'$ and $M$ as above.

Suppose now that $(X,G)$ is a topological extension of $(Y,G)$ via a factor map
$h:X\to Y$ and let $\mu$ be a $G$-invariant measure on $X$.
We say $(X,G)$ is a \emph{topo-isomorphic} extension of $(Y,G)$ with respect to
$\mu$ if $h$ is also an isomorphism with respect to $\mu$ and $h(\mu)$ where
$h(\mu)$ denotes the push-forward of $\mu$.
In this case, we call $h$ a \emph{topo-isomorphy} with respect to $\mu$.
In case that no measure is specified,  $(X,G)$ is called a \emph{topo-isomorphic}
extension of $(Y,G)$  and $h$ a \emph{topo-isomorphy} if $h\:X\to Y$ is a
topo-isomorphy with  respect to every $G$-invariant measure $\mu$ on $X$.
Observe that the push-forward of an invariant measure $\mu$ under a topo-isomorphy
is ergodic if and only if $\mu$ is ergodic.

\begin{theorem}[Mean equicontinuity and  topo-isomorphy] \label{thm:structural result mean equicontinuous systems}
The topological dynamical system $(X,G)$ is (Weyl-) mean
equicontinuous if and only  if it is a topo-isomorphic extension of
its maximal equicontinuous factor $(\mef,G)$.
\end{theorem}

Let us point out that the proof of this theorem also shows that the maximal
equicontinuous factor of a mean equicontinuous system is in a natural sense the
quotient of $X$ by the pseudometric $D$ (see the corresponding discussion in
Section \ref{subsec:mean-equi}).

The concept of topo-isomorphy is at the interface of topological and
measure-theoretical aspects of dynamical systems.
This kind of ``hybrid" notion was also recently studied by Downarowicz and Glasner
in \cite{DownarowiczGlasner2016} where a similar statement to the above is proven
for minimal dynamical systems with $G=\Z$.
We would like to mention that the direction from topo-isomorphy to mean equicontinuity
(Theorem~\ref{thm:topo-isomorphy implies mean equicontinuity})
is proven in a completely different way than in \cite{DownarowiczGlasner2016},
while the proof that mean equicontinuity implies topo-isomorphy is close to the
ones of \cite[Theorem 3.8]{LiTuYe2015} and \cite[Proposition 2.5]{DownarowiczGlasner2016},
see Section \ref{sec:topo-isomorphic extensions}.

It is very worth noting that Theorem~\ref{thm:structural result mean equicontinuous systems}
is by far not only of abstract importance but actually offers a direct way to
establish the mean equicontinuity of many well-known minimal group actions.
To emphasize this, we briefly present a (non-exhaustive) list of
minimal group actions where mean equicontinuity can always be derived by using the
structural characterization provided in Theorem~\ref{thm:structural result mean equicontinuous systems}.
Starting with $\Z$-actions, two very common example classes which are well-known to be mean equicontinuous
are Sturmian subshifts
and regular Toeplitz subshifts, see for instance \cite{Fogg2002,Kurka2003,Downarowicz2005}
for further information and references.
Non-symbolic examples can be found in the class of so-called Auslander systems (see \cite{Auslander1988}
and \cite{HaddadJohnson1997}).

Before we go beyond $\Z$-actions, we want to stress that minimal mean
equicontinuous systems are always uniquely ergodic, see
Corollary \ref{cor:structure minimal mean equicontinuous systems} (iii).
To present the reader a non-minimal and moreover, intrinsically non-uniquely ergodic
system, we provide
a symbolic $\Z$-action which has infinitely many ergodic measures in Example~\ref{ex:non-uniquely ergodic subshift}.

Concerning actions by more general groups, Theorem~\ref{thm:structural result mean equicontinuous systems}
also constitutes a basis for providing a novel and straightforward construction method
for a class of non-abelian minimal mean equicontinuous systems (outlined in
Subsection~\ref{sec:isometric subgroups}).
Moreover, we can continue the list of examples from above: higher-dimensional
subshifts, i.e., $\Z^n$-actions, which are mean equicontinuous
can for instance be obtained from regular Toeplitz arrays, see \cite{Cortez2006}.
Furthermore, the theory of quasicrystals contains many natural examples of
mean equicontinuous $\R^n$-actions, like the $\R^2$-actions obtained from Penrose
tilings \cite{Robinson1996} or the chair tiling \cite{Robinson1999}.
For more information concerning tilings and Delone sets in $\R^n$, see \cite{BaakeGrimm2013}.
It is also possible to consider Delone sets (and canonical actions induced by them) in more general groups than $\R^n$.
Especially, so-called regular model sets immediately yield mean equicontinuous group actions, see \cite{Schlottmann1999}.
%

Finally, we would like to mention that according to \cite[Corollary~5.4 (2)]{Glasner2018}, minimal tame systems are always topo-isomorphic extensions of their maximal equicontinuous factor if the corresponding acting group is amenable (see also the short discussion at the end of Section~\ref{MAP-groups}).
Systems belonging to this family are Sturmian-like $\Z^n$-actions \cite{GlasnerMegrelishvili2018}
or tame generalized Toeplitz shifts \cite{FuhrmannKwietniak2019} (see also \cite{LcackaStraszak2018} for not necessarily tame but still mean equicontinuous examples).
In fact, in \cite{FuhrmannKwietniak2019} it is shown that
every countable maximally almost periodic amenable group allows for effective mean equicontinuous minimal actions which are not equicontinuous (see also Section~\ref{MAP-groups}).

Our next main result gives a  characterization of mean
equicontinuity of a system in terms of its product system (see
Section \ref{sec:product} for details). To the authors' knowledge,
this result does not have a predecessor in any special situation.
However, it is, of course, well in line with a plethora of results
on characterizing properties of a dynamical system via properties of
its product. We need the following notion: a system $(X,G)$ is
\emph{pointwise uniquely ergodic} if the \emph{orbit closure} $\overline
{Gx}$ of every point $x\in X$ is uniquely ergodic. For such systems
we denote by $\mu_x$ the unique ergodic measure supported on the
orbit closure of $x\in X$.

\begin{theorem}[Mean equicontinuity and the product system]\label{thm:intro equivalence criterion of mean equi. involving pointwise u.e}
    The system $(X,G)$ is mean equicontinuous if and only if
    \begin{itemize}
        \item the product system $(X\times X,G)$ is pointwise uniquely ergodic
        \item and the map $(x,y)\mapsto\mu_{(x,y)}$ is continuous
            (with respect to the weak-*topology).
    \end{itemize}
\end{theorem}

As the metric $d$ is continuous on $X\times X$, the previous theorem
together with the standard result on the existence of averages of
continuous functions for uniquely ergodic dynamical systems implies
for mean equicontinuous systems that the $\limsup$ in
\eqref{def:Besicovitch-mean equicontinuous} is actually a limit and
does not depend on the chosen Følner sequence, whence, in
particular, it follows $D=D_\Fol$ for any left Følner sequence
    $\Fol$ (see also Section~\ref{sec:product} for a related discussion).

Moreover, the previous result allows us to derive the following
theorem on the independence of Følner sequences.

\begin{theorem}[Mean equicontinuity and $\mathcal{F}$-mean equicontinuity]
\label{thm:F-mean-equicontinuity implies mean-equicontinuity}
    Let $(X,G)$ be a dynamical system and assume that
    \begin{itemize}
        \item there is an invariant measure $\mu$ with full support, i.e.\ $\supp(\mu)=X$,
    \end{itemize}
    or that
    \begin{itemize}[resume]
        \item the group $G$ is abelian.
    \end{itemize}
    Then  $(X,G)$ is mean equicontinuous if and only if $(X,G)$ is
    $\Fol$-mean equicontinuous for some left Følner sequence
    $\Fol$.
\end{theorem}

Observe that if $(X,G)$ is minimal, the extra assumption of a measure with full support is evidently fulfilled.
It is noteworthy that the extra effort needed to overcome the lack of commutativity in this work is most visible
in the proof of the above statement.
We would also like to remark that in the recent article \cite{QiuZhao}, a similar statement has
independently (and by different means) been proven to hold if $G=\Z$.
Under the assumption of a minimal $\Z$-action, it is known due to \cite{DownarowiczGlasner2016}.

In the minimal case we can provide another characterization of mean
equicontinuity.
This is a characterization in terms of spectral theory, or more specifically,
in terms of a decomposition of the space $L^2 (X,\mu)$.
The corresponding proof can be found in
Section \ref{sec:mean equicontinuity and pure point spectrum}.

\begin{theorem}[Mean equicontinuity and  spectral theory]\label{thm:spectral-characterization}
    Assume $(X,G)$ is minimal.
    Then $(X,G)$ is mean equicontinuous if and only if $(X,G)$ has a unique invariant
    measure $\mu$ and $L^2 (X,\mu)$ can be written as an orthogonal sum of finite
    dimensional,  $G$-invariant subspaces consisting of continuous functions.
\end{theorem}

Now, for minimal systems we may combine all the previous theorems to
obtain a slightly simplified list of equivalent characterizations of
mean equicontinuity (note that the statements (i)-(iii) are a
generalization of Theorem 2.1 in \cite{DownarowiczGlasner2016} which
is treating $\Z$-actions).

\begin{corollary}\label{cor:structure minimal mean equicontinuous systems}
    Let $(X,G)$ be a minimal system.
    Then the following are equivalent:
    \begin{enumerate}
        \item[(i)] $(X,G)$ is mean equicontinuous.
        \item[(ii)] $(X,G)$ is $\Fol$-mean equicontinuous for some (left)
            Følner sequence $\Fol$.
        \item[(iii)] $(X,G)$ is uniquely ergodic and  topo-isomorphic to its MEF
            with respect to its unique invariant measure $\mu$.
        \item[(iv)] $(X\times X,G)$ is pointwise uniquely ergodic and
            $(x,y)\mapsto\mu_{(x,y)}$ is continuous.
            \item[(v)]  $(X,G)$ has a unique invariant
    measure $\mu$ and $L^2 (X,\mu)$ can be written as an orthogonal sum of finite
    dimensional,  $G$-invariant subspaces consisting of continuous functions.
    \end{enumerate}
\end{corollary}

We finish this section with a discussion of the spectral characterization of
mean equicontinuity when $G$ is abelian.
This case is particularly important due to its relevance for the study of
aperiodic order.
Let $\widehat{G}$ be the dual group of $G$, i.e., the
group of all continuous group homomorphisms from $G$ to the unit circle and let
$(X,G)$ be a dynamical system with an invariant probability measure $\mu$.
Then $f\in L^2 (X,\mu)$ with $f\neq 0$ is called
an \emph{eigenfunction} to the \emph{eigenvalue} $\xi \in
\widehat{G}$ if $f (g\cdot) = \xi(g) f(\cdot)$
 for all $g\in G$.
Here, the equality is understood in the sense of $L^2$ functions.
If such an $f$ is continuous with $$f(g x) = \xi (g) f(x),$$ for all $x\in
X$ and $g\in G$ it is called a \emph{continuous eigenfunction}.
The dynamical system $(X,G)$ with $G$-invariant measure $\mu$ is said to have
\emph{discrete spectrum with continuous eigenfunctions} if  there
exists an orthonormal basis for $L^2(X,\mu)$ of continuous eigenfunctions.

\begin{corollary}\label{cor:spectral-characterization-abelian-case}
    Let $G$ be abelian. A minimal system $(X,G)$ is mean equicontinuous if and
    only if it is uniquely ergodic and has discrete spectrum with continuous
    eigenfunctions.
\end{corollary}

For subshifts associated to non-periodic primitive substitutions, a classical
result by Host \cite{Host1986} states that all eigenvalues possess a continuous
eigenfunction.
Hence, the previous corollary implies that these subshifts are mean
equicontinuous if and only if they have pure point spectrum (since they are always
minimal).
This yields, for instance, that the subshifts associated to the Fibonacci and
Tribonacci substitution are mean equicontinuous.
For more information, see for instance \cite{Fogg2002} and \cite{Queffelec2010}.
Further, a generalization of Host's result to primitive tiling substitutions of
$\R^n$ with finite local complexity can be found in \cite{Solomyak2007}.

\subsubsection*{Acknowledgments}

This project has received funding from the European Union's Horizon 2020
research and innovation program under the Marie Sk\l{}odowska-Curie grant agreement
No 750865.
Furthermore, MG acknowledges support by the DFG grants JA 1721/2-1 and GR 4899/1-1.
Moreover, the authors would also like to thank Dominik Kwietniak for bringing the
example depicted in Figure \ref{fig:Gegenbeispiel} to their attention.


\section{Some basic preliminaries on ergodic theory}

In this section, we discuss some definitions and statements of the ergodic
theory of general actions by locally compact $\sigma$-compact amenable groups.
In particular, we will be concerned with averages along Følner sequences where we pay
special attention to an exposition which only requires a very fundamental set of
tools.
In particular, we will only make use of the Mean Ergodic Theorem in the following
and avoid the more sophisticated Pointwise Ergodic Theorem by Lindenstrauss
\cite{Lindenstrauss2001}.

In order to provide an alternative characterization of
topo-isomorphic extensions,
let us make the following classical measure-theoretic observation whose proof is provided
for the convenience of the reader.
\begin{proposition}\label{prop:topo-isomorphy-via-unitary-operators}
Suppose $X$ and $Y$ are compact metric spaces, $\mu$ is a Borel probability
measure on $X$ and $h\: X\to Y$ is measurable.
Then, the operator
\begin{align}\label{eq: defn unitary operator}
    U_\mu:L_2(Y,h(\mu))\to L_2(X,\mu): f\mapsto f\circ h
\end{align}
is unitary if and only if $h$ is an isomorphism $\bmod\, 0$ with respect to
$\mu$ and $h(\mu)$.
\end{proposition}
\begin{proof}
    We only show that unitarity implies isomorphy ($\!\bmod\, 0$), the
    ''if``-part is obvious.

    First, we have to fix some notation.
    For a compact metric space $(Z,d)$ we denote by $\mathcal B(Z)$ the Borel
    $\sigma$-algebra and by $\tilde{\mathcal B}(Z)$ the associated measure algebra (see, for example, \cite{Walters1982}
    for the notion of measure algebras).
    Since $U_\mu:L_2(Y,h(\mu))\to L_2(X,\mu)$
    is unitary, we can define an invertible map
    $\tilde\Phi:\tilde{\mathcal B}(Y)\to\tilde{\mathcal B}(X)$ by setting $\tilde \Phi(\tilde A)$ to be the
    equivalence class of $\Phi(A)$ in $\tilde{\mathcal B}(X)$,
    where
    \[
        U_\mu(\mathbf{1}_{A})=\mathbf{1}_{\Phi(A)},
    \]
    for $A\in\tilde A\in \tilde {\mathcal B}(Y)$.
    Note that $\Phi(A)=h^{-1}(A)$ (this also proves the well-definition of $\tilde\Phi$).
    One can check directly that $\tilde\Phi$ is a measure algebra isomorphism.
    Further, by \cite[Theorem~2.2]{Walters1982} we conclude that there
    exist sets $M\subseteq X$, $N\subseteq Y$ with $\mu(M)=h(\mu)(N)=1$
    and a Borel measurable invertible measure preserving map $\phi:M\to N$
    which induces $\tilde \Phi$, i.e.,\ $\tilde \Phi(\tilde A)=(\phi^{-1}(A\cap N))^{\sim}$ for all
    $A\in\mathcal B(Y)$, and coincides with $h$ on $M$.
    This proves the statement.
\end{proof}

Recall that any locally compact group $G$ admits a \emph{left
(right) Haar measure} (defined uniquely up to a positive
multiplicative constant) denoted by $m$ ($m_{r}$) which is
left (right) invariant, that is, for all $\varphi\in
L_1(G,m)$ and $g\in G$ we have
$\int\varphi(gs)dm(s)=\int\varphi(s)dm(s)$ ($\int\varphi(sg)dm_r(s)=\int\varphi(s)dm_r(s)$), where
$L_1(G,m)$ is the space of all Haar integrable functions on
$G$.
Note that from time to time we will also refer to the left/right Haar measure by
using the notation $\abs{\cdot}$ if there is no risk of ambiguity.

In the introduction we have already encountered the notion of a left
Følner sequence $(F_n)_{n\in\N}$ in $G$ consisting of  non-empty
compact sets in $G$ such that
    \begin{align}\label{eq:Foelner_seq}
        \lim\limits_{n\to\infty}\frac{m(gF_n\triangle F_n)}{m(F_n)}
            =0\quad\textnormal{for all }g\in G.
    \end{align}
There are also \emph{right Følner sequences} which fulfill an
analogue condition to \eqref{eq:Foelner_seq} where the left Haar
measure and the multiplication from the left is replaced by the
right Haar measure and multiplication from the right, respectively.
From now on, the standard assumption is that we deal with
left Haar measures and left Følner sequences if not stated otherwise.

Let $(X,d)$ be a compact metric space.
By $\mc C(X)$ we denote the set of all complex-valued continuous functions
on $X$ equipped with the uniform topology which is induced by
the sup norm $\Abs{\cdot}_\infty$.
Given a Borel probability measure $\mu$ on $X$ and $\varphi\in \mc C(X)$,
we set $\mu(\varphi)=\int\varphi \, d\mu$.

The next theorem is well known for $\Z$-actions and can be proven for the group
actions considered in this article by adapting the corresponding arguments from
\cite{Walters1982} and \cite{Furstenberg1981}, see also the short discussion
regarding Theorem 2.16 in \cite{MuellerRichard2013}.

\begin{theorem}\label{thm:characterisation of unique ergodicity}
    Let $(X,G)$ be a dynamical system.
    The following statements are equivalent:
    \begin{itemize}
        \item[(i)] $(X,G)$ has a unique $G$-invariant measure $\mu$.
        \item[(ii)] For each continuous function $\varphi$ on $X$ there is
            a Følner sequence $(F_n)_{n\in\N}$ with
            \[
                \lim\limits_{n\to\infty}\frac{1}{\abs{F_n}}\int\limits_{F_n}\varphi(tx)\,dm(t)=c,
            \]
            where $c$ is a constant independent of $x\in X$.
    \end{itemize}
    Further, if one of the above conditions hold, then the convergence in (ii)
    is uniform in $x\in X$, independent of the left Følner sequence $(F_n)_{n\in\N}$,
    and we have $c=\mu(\varphi)$.
\end{theorem}

For the sake of completeness, we provide a proof of the next statement.

\begin{proposition}\label{prob:right Foelner implies unique ergodicity}
    Let $(X,G)$ be a dynamical system.
    Suppose for each $\varphi \in \mc C(X)$ there is a right Følner sequence $(F_n)_{n\in\N}$
    and a constant $c\in \R$ with
    \[
        \lim\limits_{n\to\infty}\frac{1}{\abs{F_n}}\int\limits_{F_n}\varphi(tx)\,dm_r(t)=c,
    \]
    for all $x\in X$.
    Then $(X,G)$ has a unique $G$-invariant measure $\mu$ and $\mu(\varphi)=c$.
\end{proposition}
\begin{proof}
    As mentioned in the introduction, $(X,G)$ allows for a $G$-invariant measure
    $\mu$ on $X$.
    Now, using Fubini and dominated convergence, we have
    \[
        \int\limits_X\varphi\,d\mu
        =\frac{1}{\abs{F_n}}\int\limits_{F_n}\int\limits_X\varphi(tx)\,d\mu(x)dm_r(t)
        =\int\limits_X\frac{1}{\abs{F_n}}\int\limits_{F_n}\varphi(tx)\,dm_r(t)d\mu(x)
        \stackrel{n\to\infty}{\longrightarrow} c.
    \]
    Since finite Borel measures on compact metric spaces are uniquely determined
    by integrating continuous functions, we obtain that $\mu$ is the only $G$-invariant
    measure on $X$.
\end{proof}

Throughout this work, we will encounter Birkhoff averages of continuous functions,
i.e., limits of the above kind, at several places.
For that reason, we introduce the following notation: given a left Følner sequence
$\Fol$ and a continuous function $\phi$ on $X$, we set
\begin{align*}
    A_n(\Fol,\phi)(x)\=\frac{1}{\abs{F_n}}\int\limits_{F_n}\varphi(tx)\,dm(t)
\end{align*}
for $x\in X$ and $n\in \N$.
Furthermore, we introduce the following functions on $X$
\begin{align*}
    \overline{A}(\Fol,\phi)\: x\mapsto \limsup\limits_{n\to\infty}A_n(\Fol,\phi)(x)
    \quad\textnormal{and}\quad
    \underline{A}(\Fol,\phi)\: x\mapsto \liminf\limits_{n\to\infty}A_n(\Fol,\phi)(x).
\end{align*}
We simply write ${A}(\Fol,\phi)(x)$ for the above limits, provided they coincide
(as in the previous statements).
If $\Fol$ is a right Følner sequence, we refer to the analogous quantities (where the
left Haar measure is replaced by the right Haar measure) by the same symbols.

For a dynamical system $(X,G)$ with an ergodic measure $\mu$ and a left Følner
sequence $\Fol$ in $G$, we say a point $x\in X$ is \emph{($\mu$-)generic} with
respect to $\Fol$ if for every continuous function $\phi$ on $X$ the limit
${A}(\Fol,\phi)(x)$ exists and equals $\mu(\phi)$.
It is worth noting and easy to see that every $\mu$-generic point has a dense orbit
in the support of $\mu$.
For the purpose of being self-contained, we provide a proof of the next well-known
statement.
Note that a direct consequence of this statement is the well-known singularity of
ergodic measures.

\begin{theorem}\label{thm: existence of generic points}
    Let $(X,G)$ be a topological dynamical system with an ergodic measure $\mu$.
    Then every left Følner sequence $\Fol=(F_n)_{n\in\N}$ allows for a subsequence
    $\Fol'=(F'_n)_{n\in\N}$ with respect to which $\mu$-almost every point is generic.
\end{theorem}
\begin{proof}
    Let $\Fol$ be a left Følner sequence and $(\phi_\ell)_{\ell\in\N}$ be a dense
    sequence in $\mc C(X)$ (which exists due to Stone-Weierstrass).
    By the Mean Ergodic Theorem, ${A_n}(\Fol,\varphi_1)(x)\stackrel{L_1}{\longrightarrow}\mu(\phi_1)$.
    There is hence a subsequence $\Fol^{\phi_1}=(F^{\phi_1}_n)_{n\in\N}$ of $\Fol$
    such that ${A_n}(\Fol^{\phi_1},\varphi_1)(x)\to \mu(\phi_1)$ for all $x$ in a
    full measure set $X_{\phi_1}\ssq X$.
    By inductively repeating the above argument, we get that for each $\ell\in \N$
    there is a subsequence $\Fol^{\phi_{\ell+1}}$ of $\Fol^{\phi_\ell}$ such that
    ${A_n}(\Fol^{\phi_{\ell+1}},\varphi_{\ell+1})(x)\to \mu(\phi_{\ell+1})$ for all
    $x$ in a full measure set $X_{\phi_{\ell+1}}\ssq X$.
    Set $X_{\mc C(X)}=\bigcap_{\ell\in\N}X_{\phi_\ell}$ and $\Fol'=(F^{\phi_n})_{n\in\N}$.
    Clearly, we have $\mu(X_{\mc C(X)})=1$.

    Moreover, $\overline{A}(\Fol',\varphi_\ell)(x)=\underline{A}(\Fol',\varphi_\ell)(x)=\mu(\phi_\ell)$
    for all $\ell\in \N$ and $x\in X_{\mc C(X)}$.
    Note that for every fixed $x\in X$ we have that $\overline{A}(\Fol',\varphi)(x)$,
    $\underline{A}(\Fol',\varphi)(x)$, and $\mu(\phi)$ depend continuously on
    $\phi \in \mc C(X)$.
    Altogether, we thus have for every $\phi\in \mc C(X)$ and every $x\in X_{\mc C(X)}$
    that
    \begin {align*}
        \overline{A}(\Fol',\varphi)(x)=\lim_{j\to\infty}\overline{A}(\Fol',\varphi_{\ell_j})(x)
        &=\lim_{j\to\infty}\underline{A}(\Fol',\varphi_{\ell_j})(x)\\
        &=\underline{A}(\Fol',\varphi)(x)=\lim_{j\to\infty}\mu(\phi_{\ell_j})=\mu(\phi),
    \end {align*}
    where $(\phi_{\ell_j})_{j\in\N}$ is a subsequence of $(\phi_\ell)_{\ell\in\N}$
    with $\phi_{\ell_j}\to\phi$.
\end{proof}

We will need the following auxiliary statement which is immediately linked to
the ergodic representation of invariant measures, see for instance
\cite{Farrell1962} for more information.

\begin{lemma}[{\cite[Lemma 6]{Farrell1962}}]\label{lem: ergodic representation}
    Let $\mu$ be a $G$-invariant measure.
    If $\mu(A)>0$ for some Borel measurable set $A\subseteq X$, then there is
    an ergodic $G$-invariant measure $\nu$ with $\nu(A)>0$.
\end{lemma}


\section{Topo-isomorphic extensions}\label{sec:topo-isomorphic extensions}

In the following we establish the equivalence of (Weyl-) mean equicontinuity and topo-isomorphy
and thus prove our main structural result (Theorem \ref{thm:structural result mean
equicontinuous systems}).
To that end, we first gather some basics on
topo-isomorphic extensions in Subsection \ref{subsec:basic}.
Then Theorem~\ref{thm:topo-isomorphy implies mean equicontinuity} (in
Subsection~\ref{subsec:topo-isomorphy}) yields one direction
of the main theorem.
Theorem \ref{thm: mean equicontinuity implies topo-isomorphy} (in
Subsection \ref{subsec:mean-equi}) yields the other direction.

Theorem~\ref{thm:structural result mean equicontinuous systems}
naturally suggests to also look at the relation between mean equicontinuous systems and their
topo-isomorphic extensions.
We show that the preservation of the maximal equicontinuous factor is a
characteristic of such extensions (see Subsection \ref{subsec:further}).


\subsection{Basics on topo-isomorphic extensions}\label{subsec:basic}

In this section we explain the structure of topo-isomorphic extensions over
equicontinuous systems.
Roughly speaking, such systems are partitioned into uniquely ergodic
components and this will be relevant in our considerations hereafter.

\begin{proposition}\label{prop: unitarity implies distant ergodic components}
    Suppose $(X,G)$ is a topo-isomorphic extension of $(Y,G)$ via the factor
    map $h:X\to Y$.
    If $\mu_1$ and $\mu_2$ are two distinct ergodic $G$-invariant measures on $X$, then
    the image measures $h(\mu_1)$ and $h(\mu_2)$ differ as well.
\end{proposition}
\begin{proof}
    Assume for a contradiction that there exist distinct ergodic $G$-invariant
    measures $\mu_1$ and $\mu_2$ such that $h(\mu_1)=h(\mu_2)$.
    Since $h$ is a topo-isomorphy, $h(\mu_1)$ is ergodic.
    Now, consider $\mu=1/2\cdot(\mu_1+\mu_2)$.
    Clearly, $\mu$ is not ergodic, since it is a convex combination of two distinct ergodic measures.
    Since $h$ is a topo-isomorphy, $h(\mu)$ is not ergodic, too.
    This contradicts $h(\mu)=1/2\cdot(h(\mu_1)+h(\mu_2))=h(\mu_1)$.
\end{proof}

We will make use of the following classical lemma (see, for example, \cite{Auslander1988})
which gives that the notions of transitivity and minimality coincide for equicontinuous systems.
\begin{lemma}\label{lem: equicont partition}
    If $(X,G)$ is equicontinuous, then for each $x\in X$ we have that $\overline {Gx}$ is minimal.
\end{lemma}

Regarding the next statements, see also Theorem 14
(Decomposition Theorem) in \cite{Auslander1959} for the case of $\Z$-actions.
\begin{proposition}\label{prop:hilfe}
    Assume that $(X,G)$ is a topo-isomorphic extension of an equicontinuous system
    $(Y,G)$ with factor map $h$.
    \begin{enumerate}
        \item[(a)] If $A\ssq X$ is transitive, then $h(A)$ is minimal.
            In particular, if $B$ is another transitive subset of $X$, then
            either $h(A) = h(B)$ or $h(A) \cap h(B) = \emptyset$.
        \item[(b)] Let $A$ be a closed $G$-invariant subset of $X$.
            The set $h(A)$ is minimal if and only if $A$ is uniquely ergodic.
    \end{enumerate}
\end{proposition}

\begin{proof}
(a) Since factor maps preserve transitivity, this follows from Lemma~\ref{lem: equicont partition}.

\smallskip

(b) Suppose $h(A)$ is minimal.
Since $(Y,G)$ is equicontinuous,
minimal subsets are uniquely ergodic (this classical fact also follows
from Theorem~\ref{thm: unique ergodicity on orbit closures of points in supp inv measure} below).
Hence, $h$ maps every invariant measure on $A$ to the same invariant
measure on $h(A)$.
By Proposition \ref{prop: unitarity implies
distant ergodic components}, $A$ is uniquely ergodic.

Conversely, suppose $A$ is uniquely ergodic.
As any orbit closure carries an invariant measure,
$\overline {Gx}$ and $\overline{Gy}$ have a non-empty intersection for $x,y\in A$.
Now, (a) yields $h(\overline{Gx} ) = h(\overline{Gy})$
($x,y\in A$), that is, $A$ is transitive.
Again due to (a), $h(A)$ is minimal.
\end{proof}

\begin{theorem}\label{thm: decomposition theorem}
    Assume that $(X,G)$ is a topo-isomorphic extension of an equicontinuous system
    $(Y,G)$ with factor map $h:X\to Y$.
    Then the following statements are true.
    \begin{enumerate}
        \item[(a)] $(X,G)$ is pointwise uniquely ergodic.
        \item[(b)] If $\mu$ and $\nu$ are distinct ergodic
        measures on $X$ supported on transitive sets $A_\mu$ and $A_\nu$, respectively, then
        $h(A_\mu)\cap h(A_\nu)=\emptyset$.
    \end{enumerate}
\end{theorem}
\begin{proof}

(a) Clearly $\overline{Gx}$ is a transitive subset of $X$ for any $x\in X$.
The statement follows from (a) and (b) of the previous proposition.

\smallskip

(b) By (a) of the previous proposition, we either have $h(A_\mu) =
h(A_\nu)$ or $h(A_\mu) \cap h(A_\nu) = \emptyset$. So, it remains to
show that $h(A_\mu) = h(A_\nu)$ is not possible.
To that end, assume the
contrary. Then, $C:=h^{-1}(h(A_\mu)) = h^{-1}(h(A_\nu))$ is a closed
invariant subset of $X$ that contains $A_\mu$ and $A_\nu$. Hence, it
is not uniquely ergodic.
This contradicts the previous proposition.
\end{proof}

As a consequence of the preceding theorem we can decompose
topo-isomorphic extensions of equicontinuous systems into uniquely
ergodic components.
Let us introduce the following notation: whenever $(X,G)$ is a dynamical
system and $\mu$ an ergodic measure, $X_\mu$ denotes the set of all $x \in X$
whose orbit closure $\overline{Gx}$ supports $\mu$ and no
other invariant measure.
In other words, $X_\mu$ comprises the set of all points which are
$\mu$-generic with respect to each Følner sequence.

\begin{corollary}\label{coro:decomposition}
Let $(X,G)$ be a topo-isomorphic extension of an
equicontinuous system $(Y,G)$.
Then the sets $X_\mu$ partition $X$, that is $X=\bigsqcup_\mu X_\mu$,
where $\mu$ runs over all  ergodic measures on $X$.
Further, each $X_\mu$ is the
preimage of a minimal subset of $Y$
and any such preimage
coincides with an $X_\mu$.
\end{corollary}

\begin{proof}
According to the previous theorem, $(X,G)$ is pointwise uniquely ergodic which
immediately gives that the sets $X_\mu$ partition $X$.

For the second part, Lemma~\ref{lem: equicont partition} yields that it suffices
to show that to each minimal set $M\ssq Y$ there is a unique ergodic
measure $\mu$ on $X$ such that $h^{-1}(M)=X_\mu$, with $h$ the factor map from $X$ to $Y$.
Proposition~\ref{prop:hilfe}~(b) yields that there is a unique ergodic measure
$\mu$ on $X$ with $h^{-1}(M)\ssq X_\mu$.
Clearly, for any $x\in X$ whose orbit closure supports $\mu$,
we must have $\overline{Gx}\cap h^{-1}(M)\neq \emptyset$.
Now, due to Proposition~\ref{prop:hilfe}~(a), $h(\overline{Gx})$ is minimal which necessarily yields $x\in h^{-1}(M)$.
\end{proof}

\begin{corollary}
    Suppose $(X,d)$ is a compact, connected metric space and $(X,G)$ is mean equicontinuous.
    Then $(X,G)$ has either a unique ergodic measure (minimal set) or uncountably many
    ergodic measures (minimal sets).
\end{corollary}
\begin{proof}
    Recall that the support of an ergodic measure is always transitive (due to the generic points) and that every
    minimal set supports an ergodic measure.
    Due to the pointwise unique ergodicity of mean equicontinuous systems (see Theorem \ref{thm: mean equicontinuity implies topo-isomorphy}
    and Theorem \ref{thm: decomposition theorem} (a)),
    this implies that there is a one-to-one correspondence between minimal sets and ergodic measures.
    Hence, it suffices to show the statement for ergodic measures.

    Due to Corollary~\ref{coro:decomposition}, we have a bijection
    between the ergodic measures of $(X,G)$ and the minimal sets of its maximal
    equicontinuous factor $(\mef,G)$.
    By Lemma~\ref{lem: equicont partition}, $\mef$ allows for a
    partition by minimal sets which are clearly compact and
    pairwise disjoint.
    Since $\mef$ is connected (as the continuous image of $X$),
    a classical result by Sierpinski \cite{Sierpinski1918} yields that such a partition
    consists of either one or uncountably many partition elements.
\end{proof}


\subsection{Topo-isomorphy implies mean equicontinuity}\label{subsec:topo-isomorphy}

Now, we show one direction of our main structural result.
To that end, we make use of Proposition~\ref{prop:topo-isomorphy-via-unitary-operators}
and rephrase the assertion that topo-isomorphy implies mean equicontinuity as follows.

\begin{theorem}\label{thm:topo-isomorphy implies mean equicontinuity}
    Let $(X,G)$ be a dynamical system and $(\mef',G)$ an equicontinuous factor
    with factor map $\pi'$ such that
    for every $G$-invariant measure $\mu$  the operator
    \[
        U_\mu:L_2(\mef',\pi(\mu))\to L_2(X,\mu): f\mapsto f\circ\pi'
    \]
    is unitary.
    Then $(X,G)$ is mean equicontinuous and $(\mef',G)$ is the associated MEF.
\end{theorem}

Before we can turn to the proof of Theorem \ref{thm:topo-isomorphy
implies mean equicontinuity}, we need two further ingredients. The first
ingredient  is  another characterization of mean equicontinuity
which makes use of the continuous functions on $X$. For that purpose
we define the pseudometric $D_f$ associated to a function
$f  \in \mc C(X)$ by
\[
    D_f(x,y)\=\sup\left\{
    \left. \varlimsup\limits_{n\to\infty}\frac{1}{\abs{F_n}}\int_{F_n}\abs{f(tx)-f(ty)}dm(t)
    \;\right|(F_n)_{n\in\N}\textnormal{ a Følner sequence}\right\}.
\]

The following statement is well known, see
\cite[Proposition 1]{DownarowiczIwanik1988} and \cite[Theorem
2.14]{Garcia-RamosMarcus2019}. We include a proof for the
convenience of the reader.

\begin{proposition}\label{prop:equivalent char mean equicontiunity}
    The following assertions are equivalent:
    \begin{itemize}
        \item[(i)] $(X,G)$ is mean equicontinuous.
        \item[(ii)] For every $f\in \mc C(X)$ the pseudometric $D_f$
            is continuous.
    \end{itemize}
Moreover, if one of the equivalent assertions  (i) and (ii) holds, then $D(x,y) =0$
if and only if $D_f (x,y) =0$ for all continuous $f$.
\end{proposition}
\begin{proof} It is not hard to see that
the topology generated on $X$ by $D$ as well as
the mean equicontinuity of $(X,G)$ is independent of the
particular choice of the metric $d$ (provided $d$ generates the original topology on $X$).
This will be used throughout the proof.

    \smallskip

    (i)$\Rightarrow$(ii):  Observe that $d' (x,y)\=d(x,y) + |f(x) - f(y)|$ is a
    metric equivalent to $d$.
    We can hence assume w.l.o.g.\ that $|f(x) - f(y)| \leq d(x,y)$.
    This implies $D_f \leq D$.
    As $D$ is continuous, this implies (ii).

    \smallskip

    (ii)$\Rightarrow$(i): Choose a sequence $(f_n)_{n\in\N}$ of continuous functions
    on $X$ which separate points and satisfy $\|f_n\|_\infty \leq 1$ for
    all $n\in\N$.
    Then for any  $c_n>0$ with $\sum_n c_n < \infty$ we have that
    \begin{align}\label{eq:equivalent metric c_n}
        \sum_{n} c_n\abs{f_n(x)-f_n (y)}
    \end{align}
    defines a metric equivalent to $d$.
    We can hence assume w.l.o.g.\ that $d$ is given by \eqref{eq:equivalent metric c_n}.
    Now, clearly
    $D \leq \sum c_n D_{f_n}\eqqcolon\widetilde{D},$
    where $\widetilde{D}$ is continuous by (ii) and the summability of
    $(c_n)_{n\in\N}$.

    \smallskip

    The last statement has been shown along the proof.
\end{proof}

\begin{remark} The above shows that the topology on $X$ generated by $D$
agrees with the topology generated by the collection of $D_f$'s with
$f\in\mc C(X)$.
\end{remark}

The other ingredient needed for the proof of Theorem \ref{thm:topo-isomorphy
implies mean equicontinuity}  --and in some sense the main insight of the present
section-- is the following lemma.

\begin{lemma}\label{lem:insight}
    Let $(X,G)$ be a dynamical system and $(\mef',G)$ an equicontinuous
    factor with factor map $\pi'$.
    Suppose that for every $G$-invariant measure $\mu$ the operator
    \[
        U_\mu:L_2(\mef',\pi'(\mu))\to L_2(X,\mu): f\mapsto f\circ\pi'
    \]
    is unitary.
    Then, for any $f \in \mc C(X)$ and any $\eps>0$ we have $D_f (x_1,x_2)
    <\eps$ provided $\pi'(x_1)$ and $\pi' (x_2)$ are sufficiently close.
\end{lemma}
\begin{proof}
    By (a) of  Theorem \ref{thm: decomposition theorem} the orbit closure of
    $x_i$ ($i =1,2$) supports a unique ergodic measure $\mu_i$.
    W.l.o.g.\ we may assume that $\mu_1\neq\mu_2$ (if $\mu_1=\mu_2$, the following
    argument works in an analogous and slightly simplified way).
    By unitarity of the $U_{\mu_i}$'s and denseness of continuous functions
    in $L_2(\mef',\pi'(\mu_i))$, we can find $g_i \in \mc C(\mef')$ with
    \[
        \Abs{f - g_i \circ\pi'}_{L_2(X,\mu_i)}=\Abs{f - U_{\mu_i}g_i}_{L_2(X,\mu_i)}
        \leq\varepsilon/3.
    \]
    By Cauchy-Schwarz, we then obtain
    \[
        \Abs{f - g_i \circ\pi'}_{L_1(X,\mu_i)}\leq\Abs{f - g_i \circ\pi'}_{L_2(X,\mu_i)}
        \leq\varepsilon/3.
    \]
    Set $M_i=\pi'\big(\overline{Gx_i}\big)$.
    Then $M_1$ and $M_2$ are disjoint by Theorem \ref{thm: decomposition theorem} (b)
    since $\mu_1\neq \mu_2$.
    Let $S_i$ ($i=1,2$) be continuous functions
    on $\mef'$ with $\left.S_i\right|_{M_i}=1$ and $\left.S_1\right|_{M_2}=\left.S_2\right|_{M_1}=0$.
    Set $g=S_1 g_1+S_2g_2$.

    Now, for any $t\in G$ we have
    \[
        \abs{f(tx_1) - f(tx_2)}\leq\abs{f(tx_1)-g_1 \circ\pi'(t x_1)}
        +\abs{g\circ\pi'(t x_1)-g\circ\pi'(tx_2)}+\abs{g_2\circ\pi'(t x_2)-f(t x_2)}.
    \]
    Consequently, we obtain
    \[
         D_f(x_1,x_2)\leq T_1 (x_1) + D_{g\circ\pi'} (x_1,x_2) + T_2 (x_2),
    \]
    where
    \[
        T_i(x_i)\=\sup\left\{
        \left.\varlimsup_{n\to\infty}\frac{1}{\abs{F_n}}\int_{F_n}\abs{f(tx_i) - g_i\circ\pi'(t x_i)}dm(t)
        \;\right|(F_n)_{n\in\N}\textnormal{ a Følner sequence}\right\}.
    \]

    We show that all three terms become small for $x_1$ sufficiently close to $x_2$.
    By unique ergodicity on orbit closures and Theorem~\ref{thm:characterisation of unique ergodicity}, we obtain
    \[
        T_i (x_i)=\Abs{f - g_i\circ\pi'}_{L_1(X,\mu_i)}\leq\varepsilon/3.
    \]
    The term $D_{g\circ\pi'}$ can be treated as follows.
    If  $\pi'(x_1)$ is close to $\pi'(x_2)$, we obtain that $t\pi'(x_1)$ is close
    to $t\pi'(x_2)$ for all $t\in G$ (by equicontinuity).
    As $g$ is continuous (and hence uniformly continuous) on $\mef'$, this implies
    that $g\circ\pi'(tx_1) = g(t\pi'(x_1))$ is close to
    $g \circ\pi'(tx_2) = g (t\pi'(x_2))$ for all $t\in G$ and we are done.
\end{proof}

\begin{proof}[Proof of Theorem~\ref{thm:topo-isomorphy implies mean equicontinuity}]
    We first show that $(X,G)$ is mean equicontinuous.
    By Proposition \ref{prop:equivalent char mean equicontiunity}, it suffices to
    show that $D_f$ is continuous for any $f\in \mc C(X)$.
    Let such an $f$ be given and consider an arbitrary $\varepsilon >0$.
    We have to  show that if $x_1, x_2\in X$ are close, then
    $D_f(x_1,x_2)<\eps$. This, however, is clear from Lemma
    \ref{lem:insight} as for $x_1$ close to $x_2$ we clearly have
    $\pi'(x_1)$ close to $\pi'(x_2)$ due to the continuity of $\pi'$.

It remains to show that $(\mef',G)$ is the MEF.
As discussed in Subsection \ref{sec:Notation}, it suffices to show that
 $\inf_{t\in G} d(tx,ty) = 0$ whenever $\pi'(x) = \pi' (y)$. Now,
 $\pi'(x) = \pi'(y)$ implies $D_f (x,y) =0$ for all continuous $f$ on $X$ (by Lemma~\ref{lem:insight})
 and hence $D(x,y) =0$ due to Proposition \ref{prop:equivalent char mean
 equicontiunity}.
 From this and the definition of $D$ we easily find
 $\inf_{t\in G} d(tx,ty) =0$.
\end{proof}

\begin{figure}[h]
    \begin{center}
        \includegraphics[width=0.3\textwidth]{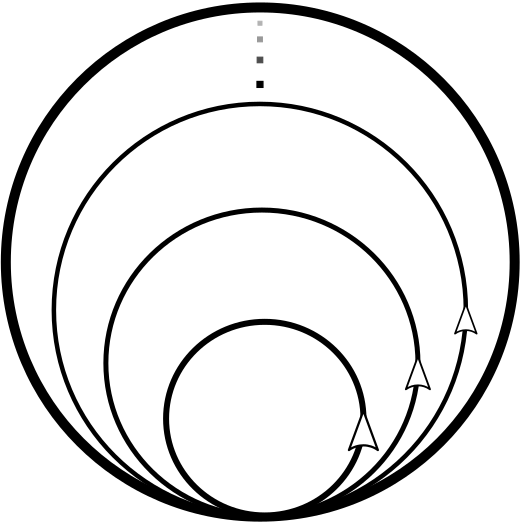}
        \caption{
            A sketch of a $\Z$-action which is not mean equicontinuous but at
            the same time a pointwise uniquely ergodic topo-isomorphic extension
            of its trivial MEF with respect to all ergodic (delta) measures.
            The continuous dynamics are as follows: points on the bold outer circle are fixed.
            Further, there are infinitely many inner circles attached to the south pole
            which accumulate at the outer circle and all points on the inner
            circles get attracted by the south pole.}
        \label{fig:Gegenbeispiel}
    \end{center}
\end{figure}

\begin{remark}
 We defined a topo-isomorphy $h$ to be a topological factor map which is
 an isomorphism $\bmod\, 0$ with respect to $\mu$ and $h(\mu)$ for every
 \emph{invariant} measure $\mu$.
 It is natural to ask whether Theorem~\ref{thm:topo-isomorphy implies mean equicontinuity} still
 remains true if we relax the assumptions on $h$
 by considering $h$ to be a factor map
 which is only an isomorphism $\bmod\, 0$ with respect to $\mu$ and $h(\mu)$ for every
 \emph{ergodic} measure $\mu$.
 In the proof of Lemma~\ref{lem:insight}, the
 topo-isomorphy with respect to every invariant measure was (implicitly) used twice:
 once, to ensure pointwise unique ergodicity and once, to
 ensure that the supports of two distinct ergodic measures have disjoint images under
 $h$ (see also Theorem~\ref{thm: decomposition theorem}).
 In fact, the latter implies the former and
 hence implies mean equicontinuity if we additionally assume $h$ to be a
 topo-isomorphy onto an equicontinuous factor with respect to every ergodic measure.
 However, it is not true that pointwise unique ergodicity
 and topo-isomorphy to an equicontinuous factor with respect
 to ergodic measures only yields that the supports of distinct ergodic measures have distinct images,
 as can be seen in Figure \ref{fig:Gegenbeispiel}.
\end{remark}


\subsection{Mean equicontinuity implies topo-isomorphy}\label{subsec:mean-equi}

In this section we establish that mean equicontinuity implies topo-isomorphy of the
dynamical system to its MEF.
Together with Theorem~\ref{thm:topo-isomorphy implies mean equicontinuity} from the
previous subsection, this proves our main structural result
Theorem \ref{thm:structural result mean equicontinuous systems}.
We first note that $D$ is actually $G$-invariant.

\begin{proposition}[Invariance of $D$]\label{prop:D-is-invariant}
Let $(X,G)$ be a dynamical system.
Then $D$ satisfies $D(tx,ty) =
D(x,y)$ for all $x,y\in X$ and $t\in G$.
\end{proposition}
\begin{proof} Recall that there exists a unique $\Delta :
G\to (0,\infty)$ (called \emph{modular function}) whose defining
property is that
$$\int h(ts) dm(t) = \Delta (s) \int h(t) dm (t),$$
for all Haar measurable $h: G\to [0,\infty)$. A short 
computation and canceling of modular functions then gives that
$$\frac{1}{\abs{Fs}} \int_{Fs} g(t) dm (t) = \frac{1}{\abs{F}} \int_F
g(ts) dm(t),$$ for all $s\in G$ and all Haar measurable bounded $g:
G\to [0,\infty)$  whenever $F$ is a compact subset of 
$G$ with positive Haar measure.
This shows that
$$D_{\Fol} (sx,sy) = D_{\Fol s} (x,y),$$ where  $\Fol s$ denotes
the sequence $(F_n s)_{n\in\N}$. Now, $(F_n s)_{n\in\N}$ is clearly a Følner sequence as
well. Hence, the desired statement follows as the definition of $D$
involves all Følner sequences.
\end{proof}

Let  a  dynamical system $(X,G)$ be given. For $x, y\in X$ write
$x\!\!\sim\! y$ if $D(x,y)=0$. If $(X,G)$ is mean equicontinuous,
then clearly the quotient map $\beta:X\to X/\!\!\sim$ is continuous.
By the invariance of $D$ due to Proposition \ref{prop:D-is-invariant}, the action
of $G$ on $X/\!\!\sim$ given by $g\beta(x)\=\beta(gx)$ is well defined and isometric.
Hence, $(X/\!\!\sim,G)$ is an equicontinuous factor of $(X,G)$.

\begin{proposition}\label{prop:Besicovitch pseudometric quotient is MEF}
    If $(X,G)$ is mean equicontinuous, then $(X/\!\!\sim,G)$ is its MEF.
\end{proposition}
\begin{proof}  As
discussed in Section \ref{sec:Notation} it suffices to show that
 $\inf_{t\in G} d(tx,ty) = 0$ whenever $\beta(x) = \beta(y)$.
 This, however, is clear.
\end{proof}

\begin{theorem}\label{thm: mean equicontinuity implies topo-isomorphy}
    Assume that $(X,G)$ is mean equicontinuous.
    Then $(X,G)$ is topo-isomorphic to its maximal equicontinuous factor $(\mef,G)$.
\end{theorem}
\begin{proof}
    Fix a $G$-invariant measure $\mu$ and let $\mu=\int\mu_z d\nu(z)$ be the
    disintegration of $\mu$ over its image measure $\nu\=\pi(\mu)$
    (see, e.g., \cite{Furstenberg1981}).
    We consider the \emph{relative product measure} $\mu\times_{\nu}\mu$
    supported in the \emph{relative product} of $X$ over $\T$
    \[
        X\times_\T X\=\{(x,y)\in X\times X\;|\;\pi(x)=\pi(y)\}
    \]
    which is defined by
    \[
        \mu\times_{\nu}\mu\=\int\mu_z\times\mu_z d\nu(z).
    \]
    Recall that $\mu\times_{\nu}\mu$ is invariant under the action of $G$ on
    $X\times_\T X$ given by $g(x,y)\=(gx,gy)$ for each $(x,y)\in X\times_\T X$ and
    $g\in G$, see Proposition 5.14 in \cite{Furstenberg1981}.

    We claim that $\mu\times_{\nu}\mu$ is only supported on the \emph{diagonal}
    $\{(x,x)\in X\times X\;|\;x\in X\}\subseteq X\times_\T X$.
    For a contradiction assume this is not the case.
    Then there exists an open set $A$ in $X\times_\T X$ which has a positive
    distance to the diagonal and fulfills $(\mu\times_{\nu}\mu)(A)>0$.
    Using Lemma \ref{lem: ergodic representation}, this yields that there is
    an ergodic measure $\tilde\mu$ on $X\times_\T X$ with $\tilde\mu(A)>0$.
    According to Theorem \ref{thm: existence of generic points},
    $\tilde\mu$-almost every point is $\tilde\mu$-generic with respect to some
    Følner sequence $\Fol$.
    Now, for every such
    $(x,y)\in X\times_\T X$ we have
    \[
        D(x,y)\geq D_\Fol(x,y)=\int d(z,w)d\tilde\mu(z,w)>0.
    \]
    This is in contradiction to the previous proposition because
    $\pi(x)=\pi(y) \Leftrightarrow D(x,y)=0$ in case that $(X,G)$ is mean
    equicontinuous.

    Now, observe that the only measures supported in $\pi^{-1}(z)$ whose Cartesian
    squares are supported in the diagonal of $X\times_\T X$ are delta measures.
    Thus, $\mu_z$ is a delta measure for $\nu$-almost every $z\in\T$.
    Finally, the map which assigns to each $z$ the support of $\mu_z$ is an
    isomorphism with respect to $\nu$ and $\mu$ whose inverse coincides with $\pi$ for
    $\mu$-a.e.\ point.
\end{proof}

\subsection{Further properties and first non-minimal examples}\label{subsec:further}

Here, we discuss first consequences of the results of the previous subsections.
In particular, we show that the preservation of the maximal equicontinuous
factor is a characteristic feature of topo-isomorphic extensions of
mean equicontinuous systems.
Furthermore, we discuss some examples of non-minimal mean
equicontinuous systems for $G=\Z$.

\begin{theorem}[Characterization of mean equicontinuous extensions]
\label{thm:characterisation of mean equicontinuous extensions via same mef}
    Let $(X,G)$ be an extension of a mean equicontinuous system $(Y,G)$.
    Then, $(X,G)$ is topo-isomorphic to $(Y,G)$ if and only if it is mean
    equicontinuous and its MEF agrees with that of $(Y,G)$.
\end{theorem}
\begin{proof}
    Assume first that  $(X,G)$ is a topo-isomorphic extension of $(Y,G)$.
    By Theorem \ref{thm: mean equicontinuity implies topo-isomorphy},
    $(Y,G)$ is a topo-isomorphic extension of its MEF $(\mef,G)$.
    Clearly, $(X,G)$ is also a topo-isomorphic extension of $(\mef,G)$.
    The statement now follows from Theorem \ref{thm:topo-isomorphy implies mean
    equicontinuity}.

    Consider now the situation that $(X,G)$ is mean equicontinuous and
    the MEFs of $(Y,G)$ and $(X,G)$ agree.    Let $h$ be the factor map
    from $X$ to $Y$ and $\pi$ a factor map from $Y$ to $\mef$ (the MEF of both systems).
    Note that both $\pi$ and $\pi\circ h:X\to\mef$ are topo-isomorphies,
    according to Theorem \ref{thm: mean equicontinuity implies topo-isomorphy}.
    This implies that $h$ is a topo-isomorphy, too.
\end{proof}

\begin{corollary}
    If two  equicontinuous dynamical systems $(X,G)$ and $(Y,G)$ are topo-isomorphic,
    then they are in fact topological conjugate.
\end{corollary}
\begin{proof} By the previous theorem such systems share the same
MEF. By equicontinuity, however, they agree with their MEF.
\end{proof}

\begin{remark}
    The corollary is reminiscent of the rigidity phenomenon which is well known
    for ergodic abelian equicontinuous group actions, see for instance \cite{FeresKatok2002}.
\end{remark}

In order to state another consequence of Theorem
\ref{thm:characterisation of mean equicontinuous extensions via same
mef} we need the following observation.

\begin{proposition} Any topological factor of a mean equicontinuous system is
mean equicontinuous as well.
\end{proposition}
\begin{proof}
    Let $(Y,G)$ be a factor of a mean equicontinuous system $(X,G)$ with
    factor map $h$.
    By Proposition \ref{prop:equivalent char mean equicontiunity}, it suffices
    to show that $D_f$ is continuous for any continuous $f$ on $Y$.
    As the factor map $h$ is a continuous surjective map between compact spaces,
    continuity of $D_f$ is  equivalent to continuity of $D_f \circ h = D_{f\circ h}$
    and the statement follows from a further application of
    Proposition~\ref{prop:equivalent char mean equicontiunity}.
\end{proof}

Given this proposition,  Theorem \ref{thm:characterisation of mean
equicontinuous extensions via same mef} has the following immediate
consequence (systems fitting into the setting of the following statement can
be found in \cite[Section 5]{DownarowiczDurand2002}).

\begin{corollary}\label{cor: 1} If  $(X,G)$ is mean
    equicontinuous and an extension of $(Y,G)$ with the same MEF,
    then $(X,G)$ is a topo-isomorphic extension of $(Y,G)$.
\end{corollary}

Much of the previous work on mean equicontinuity is concerned with minimal
$\Z$-actions.
Therefore, we would like to close this section with a discussion of two simple kinds of
examples of well-known systems
which are mean equicontinuous but
not minimal.
The first example will be still transitive (in fact, as we will see, all but one point
have a dense orbit) and the second kind of examples will have no dense orbits
but will still be uniquely ergodic.
For a non-uniquely ergodic system, see Example~\ref{ex:non-uniquely ergodic subshift}.

\begin{example}\label{ex:Cantor substitution}
    Consider the \emph{Cantor substitution}
    \[
        0\mapsto 010\quad\textnormal{and}\quad1\mapsto 111.
    \]
    For a general introduction to substitution systems, see for example \cite{Kurka2003}.
    There are two infinite sequences in $\{0,1\}^\N$ which are invariant with
    respect to the Cantor substitution: the constant sequence $(111\ldots)$
    and the sequence $\omega$ obtained by applying the substitution successively
    to the letter $0$ and its images, i.e.,
    \[
        0 \mapsto 010 \mapsto 010111010 \mapsto 010111010 111111111 010111010\mapsto\cdots.
    \]

    Now, there is a standard method to obtain a two-sided subshift $(\Sigma_\omega,\sigma)$
    from $\omega$, see for instance \cite[Proposition 3.71]{Kurka2003}.
    That is, $\Sigma_\omega$ is a closed subset of $\{0,1\}^\Z$ (equipped with the
    product topology) which is invariant under the action of the left shift
    $\sigma:\{0,1\}^\Z\to\{0,1\}^\Z$.
    By making use of the concrete structure of $\omega$, it is not difficult to
    see that all points in $\Sigma_\omega$, except the constant sequence
    $(\ldots111\ldots)$, have a dense orbit and that the letter $0$ occurs with
    zero density in each sequence of $\Sigma_\omega$.
    The former implies that $\Sigma_\omega$ is uncountable and the latter that
    $(\Sigma_\omega,\sigma)$ is uniquely ergodic, with the unique invariant
    measure the delta measure supported on $(\ldots111\ldots)$.
    Thus $(\Sigma_\omega,\sigma)$ is a non-trivial topo-isomorphic
    extension of its trivial MEF and hence, mean equicontinuous.
\end{example}

\begin{example}\label{ex:Denjoy homeomorphisms}
    By a classical result of Denjoy \cite{Denjoy1932}, there exist examples of
    $C^1$ circle diffeomorphisms which have a rigid rotation $(\S^1,R_{\alpha})$,
    with $\alpha\in\R$ irrational, as a factor but are not conjugate to it.
    Herman \cite{Herman1979} showed later that these examples can even be made
    $C^{1+\varepsilon}$ for any $\varepsilon<1$.
    We will refer to these kind of systems as {\em Denjoy examples}.

    All Denjoy examples have a unique minimal set $C\subset\S^1$ and a unique
    invariant measure $\mu$ supported on $C$.
    We claim that any Denjoy system $(\S^1,f)$ is mean equicontinuous because of
    the following reason.
    Since the factor map $\pi:\S^1\to\S^1$, extending $(\S^1,R_\alpha)$ to $(\S^1,f)$,
    is monotone, we have that $\pi^{-1}(\theta)$ for $\theta\in\S^1$ is either a
    singleton or an interval.
    This immediately implies that the set of non-invertible points
    $\{\theta\in \S^1\: \#\pi^{-1}(\theta)>1\}$ is countable.
    Accordingly, we get that $h$ is invertible on a full measure set with respect
    to $\mu$ (since $h(\mu)$ is the
    Lebesgue measure on $\S^1$, the unique invariant measure of $R_\alpha$).

    McSwiggen has shown that there are Denjoy homeomorphisms on higher-dimensional
    tori that share the same properties just mentioned, in particular, that the
    set of non-invertible points is countable.
    This means these systems are mean equicontinuous, too.
    For examples on the two-torus, see \cite{McSwiggen1993} as well as
    \cite{NortonVelling1994,NortonSullivan1996} for more information concerning these systems.
    For examples defined on general $k$-tori, $k\geq 2$, see \cite{McSwiggen1995}.
\end{example}


\section{Mean equicontinuity via product systems}\label{sec:product}

In Theorem \ref{thm: decomposition theorem} we have seen that any
mean equicontinuous system is pointwise uniquely ergodic. Here, we
show that pointwise unique ergodicity of the product system together
with a continuity property is an equivalent characterization of mean
equicontinuity.

\begin{proposition}\label{prop:continuity-of-overline-A-and-underline-A}
    If there is a left (or right) Følner sequence $\Fol=(F_n)_{n\in\N}$ so that $(X,G)$
    is $\Fol$-mean equicontinuous, then for each $\varphi \in \mc C(X)$,
    $\overline{A}(\Fol,\varphi)(\cdot)$ and $\underline{A}(\Fol,\varphi)(\cdot)$
    are continuous.
\end{proposition}
\begin{proof}
    Similarly as in the proof of Proposition~\ref{prop:equivalent char mean equicontiunity}
    (i)$\Rightarrow$(ii), we see that for each
    $\eps>0$ there is $\delta>0$ such that
    \[
        \varlimsup_{n\to \infty}\frac{1}{\abs{F_n}}
            \int\limits_{F_n}\abs{\varphi(tx)-\varphi(ty)}dm(t)<\eps,
    \]
    whenever $d(x,y)<\delta$.
    This immediately gives the continuity of $\overline{A}(\Fol,\varphi)$ and
    $\underline{A}(\Fol,\varphi)$.
\end{proof}

In the following, if $(X,G)$ is pointwise uniquely ergodic, then the map
$x\mapsto\mu_x$ from $X$ into the space of all Borel probability
measures on $X$ (equipped with the weak-*topology) is defined to send each
$x\in X$ to the unique $G$-invariant measure $\mu_x$ supported on
$\overline{Gx}$.

\begin{theorem}\label{thm:equivalence criterion of mean equi. involving pointwise u.e}
    For a system $(X,G)$ the following conditions are equivalent:
    \begin{itemize}
        \item[(i)] $(X,G)$ is mean equicontinuous.
        \item[(ii)] $(X\times X,G)$ is pointwise uniquely ergodic
            and the map $(x,y)\mapsto\mu_{(x,y)}$ is continuous.
    \end{itemize}
\end{theorem}
\begin{proof}
    First, assume that $(X,G)$ is mean equicontinuous.
    This easily implies
    that the product system $(X\times X,G)$ is also mean equicontinuous.
    According to Theorem \ref{thm: decomposition theorem} (a), this in turn yields
    that $(X\times X,G)$ is pointwise uniquely ergodic.
    Now, consider some left Følner sequence $\Fol$.
    From the previous proposition    we have that for every
    $\varphi \in \mc C(X\times X)$   the functions
    $\underline A(\Fol,\varphi)(\cdot)$ and $\overline A(\Fol,\varphi)(\cdot)$ are
    continuous on $X\times X$.
    Hence, using Theorem \ref{thm:characterisation of unique ergodicity}, for a
    sequence of points $(x_n,y_n)\in X\times X$ converging to $(x,y)$ as $n\to\infty$
    we have
    \[
        \lim_{n\to\infty}\int\varphi\,d\mu_{(x_n,y_n)}
        =\lim_{n\to\infty}A(\Fol,\varphi)(x_n,y_n)
        =A(\Fol,\varphi)(x,y)=\int\varphi\,d\mu_{(x,y)}.
    \]
    Since $\varphi\in\mc C(X\times X)$ was arbitrary,
    $\mu_{(x_n,y_n)}$ converges weakly to $\mu_{(x,y)}$ as $n\to\infty$ .

    Regarding the opposite direction, consider $(x_n,y_n)\in X\times X$, $n\in\N$
    converging to $(x,y)$ as $n\to\infty$.
    By using Theorem \ref{thm:characterisation of unique ergodicity} again, observe that
    for each left Følner sequence $\Fol$
    \[
        \lim_{n\to\infty}D_\Fol(x_n,y_n)=\lim_{n\to\infty}\int d(z,w)\,d\mu_{(x_n,y_n)}(z,w)
        =\int d(z,w)\,d\mu_{(x,y)}(z,w)=D_\Fol(x,y).
    \]
    Therefore, $\lim_{n\to\infty}D(x_n,y_n)=D(x,y)$ and $(X,G)$ is mean equicontinuous.
\end{proof}

Let us conclude with a few comments on the natural question of why we have to formulate the assumptions
of Theorem~\ref{thm:equivalence criterion of mean equi. involving pointwise u.e}~(ii) for the product system.
Obviously, a system is
automatically pointwise uniquely ergodic if its product system has
this property (if additionally $(x,y)\mapsto\mu_{(x,y)}$ is
continuous, then $x\mapsto\mu_x$ is continuous as well).
However, the converse is not true.
For example, the product of a uniquely
ergodic weakly mixing system with itself is ergodic with respect to
the product measure. Hence, there are points whose orbit is dense in
the full product and hence supports the product measure as well as
the diagonal measure.

Furthermore, the next example shows that pointwise unique ergodicity
of the product system (and hence, of the original system) and
continuous dependence of the map $x\mapsto \mu_x$ does not imply continuity of
the map $(x,y)\mapsto\mu_{(x,y)}$.

\begin{example}
 Let $C\ssq \S^1$ be a Cantor set which does not contain rationals\footnote{For
 example, one may take a sufficiently small cover $U$ of the rationals
 (e.g., a cover $U$ with $m_{\S^1}(U)<1$) and set
 $C = U^c\setminus \mc I$ where $\mc I=\{x\in U^c\:\text{ there is } \eps>0
 \text{ such that } B_\eps(x)\cap U^c \text{ is at most countable}\}$.}
 and consider the skew-product $F\:C\times \S^1\to C\times \S^1:(x,\theta)\mapsto (x,\theta+x)$.
 Clearly, the corresponding $\Z$-action is pointwise uniquely ergodic (the unique
 invariant measure supported on the orbit closure of $(x,\theta)$ is given by
 $\mu_{(x,\theta)}=\delta_x\times m_{\S^1}$) and the map
 $(x,\theta)\mapsto \mu_{(x,\theta)}$ is continuous.
 Furthermore, the product system is topologically conjugate to
 \[
    \hat F\:C\times C\times \S^1\times \S^1\to C\times C\times \S^1\times \S^1\:
    (x_1,x_2,\theta_1,\theta_2)\mapsto (x_1,x_2,\theta_1+x_1,\theta_2+x_2)
 \]
 and still pointwise uniquely ergodic.
 However, the map $(x_1,x_2,\theta_1,\theta_2)\mapsto \mu_{(x_1,x_2,\theta_1,\theta_2)}$ cannot be continuous as this would
 imply mean equicontinuity (due to Theorem~\ref{thm:equivalence criterion of mean equi. involving pointwise u.e}) while
 the MEF of $(C\times \S^1,F)$ coincides with the identity on $C$ so that the corresponding factor map is not a topo-isomorphy.

 In fact, this can be seen explicitly: if $x_1,x_2\in C$ are rationally independent,
 then $\mu_{(x_1,x_2,\theta_1,\theta_2)}= \delta_{x_1}\times \delta_{x_2}\times m_{\S^2}$
 (independently of $\theta_1$ and $\theta_2$).
 This is, of course, true for a dense set of points in $C\times C\times \S^1\times\S^1$.
 However, given any $x_0\in C$, we clearly have
 $\mu_{(x_0,x_0,\theta_1,\theta_2)}=\delta_{x_0}\times \delta_{x_0} \times m_{\theta_1,\theta_2}$,
 where $m_{\theta_1,\theta_2}$ denotes the (one-dimensional) Lebesgue measure on
 the set $\{(x+\theta_1,x+\theta_2)\: x\in \S^1\}\ssq \S^2$.
 Obviously, $(x_1,x_2,\theta_1,\theta_2)\mapsto\mu_{(x_1,x_2,\theta_1,\theta_2)}$
 is not continuous in $(x_0,x_0,\theta_1,\theta_2)$ for $\theta_1,\theta_2\in \S^1$.
\end{example}

\section{Relating Besicovitch- and Weyl-mean equicontinuity}\label{sec:relation Besicovitch and Weyl}

By its very definition, Weyl-mean equicontinuity is a stronger
assumption than Besicovitch-$\Fol$-mean equicontinuity.
Quite remarkably, it turns out that Besicovitch-$\Fol$-mean equicontinuity, i.e.,
control over one Følner sequence $\Fol$ suffices to conclude Weyl-mean equicontinuity
in many situations.
A detailed study is given in this section and the presented results yield a proof
of Theorem \ref{thm:F-mean-equicontinuity implies mean-equicontinuity}.
By means of this result, we provide a non-trivial non-uniquely ergodic mean equicontinuous systems
at the end of this section.

In the following, we will also speak of $\Fol$-mean
equicontinuity with respect to a right Følner sequence $\Fol$,
where the definition is completely analogous to the definition using
left Følner sequences given in \eqref{def:Besicovitch-mean
equicontinuous}.

\begin{theorem}\label{thm:mean-equicontinuity on recurrent dynamics}
    Let $(X,G)$ be $\Fol$-mean equicontinuous for some left Følner sequence
    $\Fol$ and let there be a $G$-invariant measure $\mu$ with $\supp(\mu)=X$.
    Then $(X,G)$ is mean equicontinuous.
\end{theorem}

\begin{remark} A  comment on the assumption $\supp (\mu) =X$ may be in
order.
As is well known, every dynamical system $(X,G)$ possesses a
$G$-invariant measure $\mu$ of
\emph{maximal support}, that is, a measure $\mu$ such that $\supp(\mu)$ contains the
support of any other $G$-invariant measure. This support is
clearly unique and coincides with the closure of the union of all supports of ergodic
measures.
While in general, $\supp(\mu)$ may not fill the whole space $X$,
we can, of course, restrict attention to the maximal support and
then apply the above theorem.
By the Poincaré Recurrence Theorem, one may think of this as restricting to
the recurrent dynamics of the system $(X,G)$.
\end{remark}

Recall that for minimal dynamical systems every invariant measure
has full support.

\begin{corollary}
    If $(X,G)$ is minimal, then $\Fol$-mean equicontinuity
    for some left Følner sequence $\Fol$ implies mean equicontinuity.
\end{corollary}

Next we will collect some further assertions needed for the proof of
Theorem~\ref{thm:mean-equicontinuity on recurrent dynamics}.
The following elementary lemma makes up for the (possible) lack of separability of $G$.
Recall that $G$ is assumed to be $\sigma$-compact.

\begin{lemma}\label{lem: sort of separability}
 Let $(X,G)$ be a dynamical system.
 Then there exists a countable subgroup $T\leq G$ such that
 $\overline{Tx}=\overline{Gx}$ for every $x\in X$.
\end{lemma}
\begin{proof}
 Since $G$ is $\sigma$-compact, there exists an exhausting sequence $(K_n)_{n\in\N}$
 of compact subsets of $G$.
 Given $\eps>0$, set $T_n\ssq K_n$ to be a finite subset such that for each
 $s\in K_n$ there is $t\in T_n$ with $\sup_{x\in X}d(sx,tx)<\eps$.
 Note that $T_n$ is well defined due to the continuity of the defining action of
 $(X,G)$ as well as the compactness of $K_n$ and $X$.
 Set $T^\eps\=\bigcup_{n\in\N} T_n$.
 Then $T'\=\bigcup_{n\in \N} T^{1/n}$ is countable and verifies
 $\overline{T'x}=\overline{Gx}$ for every $x\in X$.
 Letting $T$ be the group generated by $T'$ proves the statement.
\end{proof}

\begin{proposition}\label{prop: support of ergodic measure is uniquely ergodic}
    Suppose $(X,G)$ is $\Fol$-mean equicontinuous with respect to some  left
    Følner sequence $\Fol$.
    Then the support of each ergodic measure $\mu$ is uniquely ergodic.
\end{proposition}
\begin{proof}
    By possibly restricting to the support of $\mu$, we may assume without loss
    of generality that $X=\supp(\mu)$.
    By possibly going over to a subsequence of $\Fol$, we may further assume without loss
    of generality that there is a full measure set $X_\mu$ of $\mu$-generic points with respect to $\Fol$
    (see Theorem~\ref{thm: existence of generic points}).

    From Proposition \ref{prop:continuity-of-overline-A-and-underline-A}
    we  know that for each $\varphi\in\mc C(X)$ the maps
    $\overline{A}(\Fol,\varphi)(\cdot)$ and $\underline{A}(\Fol,\varphi)(\cdot)$
    are continuous.
    Hence, with $T$ as in Lemma~\ref{lem: sort of separability} and $x_0\in \bigcap_{t\in T} tX_\mu$
    we have that
    \[
        \overline{A}(\Fol,\varphi)(x)=\underline{A}(\Fol,\varphi)(x)=\mu(\phi),
    \]
    for all $x$ from the set $Tx_0\ssq X_\mu$.
    Note that $Tx_0$ is dense because of Lemma~\ref{lem: sort of separability} and the fact
    that $\mu$-generic points are transitive.
    By the continuity of $\overline{A}(\Fol,\varphi)(\cdot)$ and
    $\underline{A}(\Fol,\varphi)(\cdot)$, we get that ${A}(\Fol,\varphi)(x)$, in fact,
    exists and coincides with $\mu(\phi)$ for all $x\in X$.
    As $\varphi\in \mc C(X)$ was arbitrary, Theorem \ref{thm:characterisation of unique ergodicity}
    yields the unique ergodicity of $(X,G)$.
\end{proof}

\begin{theorem}\label{thm: unique ergodicity on orbit closures of points in supp inv measure}
    Suppose $(X,G)$ is $\Fol$-mean equicontinuous with respect to some  left
    Følner sequence $\Fol$.
    Consider a point $x\in \supp(\mu)$ where $\mu$ is an arbitrary $G$-invariant measure.
    Then the orbit closure $\overline{Gx}$ is uniquely ergodic.
\end{theorem}
\begin{proof}
    By Theorem~\ref{thm:characterisation of unique ergodicity}, it suffices to
    show that ${A}(\Fol,\varphi)(\cdot)$ exists and is constant on $\overline{Gx}$
    for each $\phi\in \mc C\big(\overline{Gx}\big)$.
    In fact, by Tietze's Extension Theorem, it is enough to consider $\phi \in \mc C(X)$.
    Observe that by Lemma \ref{lem: ergodic representation} there is a sequence
    $(x_n)_{n\in\N}$ in $X$ with $x_n\to x$ for $n\to\infty$ such that each
    $x_n$ lies in the support of an ergodic measure.
    By Proposition~\ref{prop:continuity-of-overline-A-and-underline-A},
    the functions $\overline{A}(\Fol,\varphi)(\cdot)$ and
    $\underline{A}(\Fol,\varphi)(\cdot)$  and hence, $\overline{A}(\Fol,\varphi)(g\, \cdot)$
    and $\underline{A}(\Fol,\varphi)(g\, \cdot)$ are continuous for every $g\in G$,
    so that
    \[
        \overline{A}(\Fol,\varphi)(gx)=\lim_{n\to\infty} \overline{A}(\Fol,\varphi)(gx_n)=
        \lim_{n\to\infty} \underline{A}(\Fol,\varphi)(gx_n)=\underline{A}(\Fol,\varphi)(gx),
    \]
    where we used the unique ergodicity on ergodic components
    (Proposition~\ref{prop: support of ergodic measure is uniquely ergodic})
    in the second equality.
    This proves equality of $\overline{A}(\Fol,\varphi)(\cdot)$ and
    $\underline{A}(\Fol,\varphi)(\cdot)$ on $Gx$.
    Similarly, we see that $\overline{A}(\Fol,\varphi)(\cdot)$ and
    $\underline{A}(\Fol,\varphi)(\cdot)$ are constant on $Gx$.
    As both functions are continuous, this shows that ${A}(\Fol,\varphi)(\cdot)$
    exists and is constant on $\overline{Gx}$ for each $\phi\in\mc C(X)$.
\end{proof}

\begin{proof}[Proof of Theorem \ref{thm:mean-equicontinuity on recurrent dynamics}]
    By Theorem \ref{thm:equivalence criterion of mean equi. involving pointwise u.e},
    it suffices to show  that $(X\times X,G)$ is pointwise uniquely ergodic and that the
    map $(x,y)\mapsto \mu_{(x,y)}$ is continuous.
    To that end, we first note that with $(X,G)$ the product system
    $(X\times X,G)$ is $\Fol$-mean equicontinuous as well.
    Moreover, by the assumptions, the measure $\mu\times\mu$ has full support on $X\times X$
    which implies that $(X\times X,G)$ is pointwise uniquely ergodic, by
    Theorem \ref{thm: unique ergodicity on orbit closures of points in supp inv measure}.

    It remains to show the continuity of the map $(x,y)\mapsto\mu_{(x,y)}$.
    By pointwise unique ergodicity and Theorem \ref{thm:characterisation of unique ergodicity},
    we have for any $\varphi  \in\mc C(X\times X)$ that
    \[
        \mu_{(x,y)}(\phi)=\overline{A}(\Fol,\varphi)(x,y).
    \]
    Hence, the continuity follows from Proposition
    \ref{prop:continuity-of-overline-A-and-underline-A} applied to $(X
    \times X,G)$.
\end{proof}

In Theorem \ref{thm:mean-equicontinuity on recurrent dynamics} we
had to assume full support of the measure  to deduce mean
equicontinuity from $\Fol$-mean equicontinuity for some left Følner
sequence $\Fol$. This is not needed if we know that a system is
$\Fol$-mean equicontinuous for a right Følner sequence $\Fol$.
Details are discussed next.

\begin{proposition}\label{prop:transitivity implies unique ergodicity}
    Let $\Fol=(F_n)_{n\in\N}$ be a right Følner sequence so that $(X,G)$ is
    $\Fol$-mean equicontinuous.
    If $(X,G)$ is transitive, then it has a unique $G$-invariant measure $\mu$ and
    there is a subsequence $\Fol'=(F'_n)_{n\in\N}$ of $\Fol$ such that
    \[
        \lim\limits_{n\to\infty} \frac{1}{\abs{F'_n}}\int\limits_{F'_n}\varphi(tx)\,dm_r(t)
        =\mu(\varphi)\quad(x\in X),
    \]
    for each $\varphi\in\mc C(X)$.
\end{proposition}
\begin{proof} Given $\phi\in \mc C(X)$, we know by
    Proposition~\ref{prop:continuity-of-overline-A-and-underline-A} that
    the maps $\overline{A}(\Fol,\varphi)(\cdot)$ and
    $\underline{A}(\Fol,\varphi)(\cdot)$ are continuous.
    Moreover, as $\Fol$ is a right Følner sequence, $\overline{A}(\Fol,\varphi)$
    and $\underline{A}(\Fol,\varphi)$ are invariant and hence--due to the
    transitivity of $(X,G)$--constant.

    Now, by the Stone-Weierstrass Theorem, $\mc C(X)$ is separable
    so that there exists a dense sequence of functions
    $(\varphi_n)_{n\in\N}$ in $\mc C(X)$.
    Observe that there is a subsequence $\Fol^1$ of $\Fol$ with
    $\overline{A}(\Fol^1,\varphi_1)=\underline{A}(\Fol^1,\varphi_1)$.
    Recursively, we obtain a subsequence $\Fol^{n+1}$ of $\Fol^n$ with
    $\overline{A}(\Fol^{n+1},\varphi_{n+1})=\underline{A}(\Fol^{n+1},\varphi_{n+1})$
    for each $n\in\N$.
    By setting $\Fol'=(\Fol^n_n)_{n\in\N}$, we eventually have a right Følner
    sequence $\Fol'$ with respect to which
    $\overline{A}(\Fol',\varphi_n)=\underline{A}(\Fol',\varphi_n)$
    for all $n\in \N$.
    As $\overline{A}(\Fol,\varphi)$ and $\underline{A}(\Fol,\varphi)$ depend
    continuously on $\varphi$, we have
    \[
        \overline{A}(\Fol',\varphi)=\underline{A}(\Fol',\varphi)=\text{const},
    \]
    for all $\varphi\in\mc C(X)$.
    By using Proposition \ref{prob:right Foelner implies unique ergodicity},
    we obtain the desired statement.
\end{proof}

\begin{theorem}\label{thm:mean-equicontinuity and right Foelner sequences}
    If  $(X,G)$ is $\Fol$-mean equicontinuous for some right Følner
    sequence $\Fol$, then $(X,G)$ is mean equicontinuous.
\end{theorem}
\begin{proof}
Given the preceding result, the proof is almost literally the same
as the one of Theorem \ref{thm:mean-equicontinuity on recurrent dynamics}.
\end{proof}

Recall that a Følner sequence $\Fol$ is \emph{two-sided} if it is a left and right
Følner sequence (this implies in particular that $G$ is unimodular).
Clearly, if $G$ is abelian, every Følner sequence is two-sided.
We immediately obtain the following corollaries.

\begin{corollary}
    Suppose $G$ is unimodular and let $(X,G)$ be $\Fol$-mean equicontinuous for
    a two-sided Følner sequence $\Fol$.
    Then $(X,G)$ is mean equicontinuous.
\end{corollary}

\begin{corollary}\label{cor:abelian}
    If $G$ is abelian and $(X,G)$ is $\Fol$-mean equicontinuous for some Følner
    sequence $\Fol$, then $(X,G)$ is mean equicontinuous.
\end{corollary}

\begin{proof}[Proof of Theorem
\ref{thm:F-mean-equicontinuity implies mean-equicontinuity}] This
theorem is now an immediate consequence of Theorem
\ref{thm:mean-equicontinuity on recurrent dynamics} and Corollary
\ref{cor:abelian}.
\end{proof}

Last, we would like to address the question of whether there are non-trivial non-uniquely ergodic
mean equicontinuous systems (that is, non-uniquely ergodic mean equicontinuous systems
which are neither finite unions of uniquely ergodic systems nor products of such).
The following example demonstrates that such non-trivial neither minimal nor uniquely
ergodic systems exist.

\begin{example}\label{ex:non-uniquely ergodic subshift}
 Given a sequence $x=(x_k)_{k\in\Z}\in \{0,1\}^\Z$ and $p\in \N$, let us set
 the \emph{$p$-periodic part} of $x$ to be $\Per(x,p)\=\{k\in\Z\;|\;x_k=x_{k+np} \ (n\in \Z) \}$.
 We put $\mc T$ to be the closure of
 \[
    \mc T'\=\left\{x\in \{0,1\}^\Z\;|\;\emptyset\neq \Per(x,2^n)\subsetneq \Per(x,2^{n+1}) \ (n\in \N)\right\}
 \]
 in $\{0,1\}^\Z$ (equipped with the product topology).
 Observe that for every $x\in \mc T'$ and each $n\in \N$, we have that there is
 exactly one $k\in[0,2^n-1]\setminus\Per(x,2^n)$.

 Clearly, $\mc T$ is $\sigma$-invariant where $\sigma:\{0,1\}^\Z\to\{0,1\}^\Z$
 denotes the left shift.
 We show that $(\mc T,\sigma)$ is mean equicontinuous by proving that it is
 $\Fol$-mean equicontinuous for $\Fol=([0,2^{n}-1])_{n\in \N}$, see Corollary \ref{cor:abelian}.
 To that end, define $D^n$ to be the pseudometric given  by
 \[
    D^n(x,y)\=1/2^n\cdot \sum_{\ell=0}^{2^n-1}d(\sigma^\ell(x),\sigma^\ell(y)),
 \]
 where we consider $d$ to be the \emph{Cantor metric}
 with $d(x,y)\=2^{-\min\{|k|\;|\;k\in \Z \text{ and } x_k\neq y_k\}}$.
 By definition, $\limsup_{n\to\infty} D^n(x,y)=D_\Fol(x,y)$ for $x,y\in \{0,1\}^\Z$.

 Now, given $x_1,x_2\in \mc T'$ with $d(x_1,x_2)\leq 2^{-2^n}$, observe that there are
 at most two elements in $[0,2^n-1]\setminus (\Per(x_1,2^n)\cup \Per(x_2,2^n))$
 so that
 \begin{align*}
 D^k(x_1,x_2)&=1/2^{k-n}\cdot\sum\limits_{m=0}^{2^{k-n}-1} D^n(\sigma^{m\cdot 2^n}(x_1),\sigma^{m\cdot 2^n}(x_2))\\
 &\leq \max_{m=0,\ldots,2^{k-n}-1} D^n(\sigma^{m\cdot 2^{n}}(x_1),\sigma^{m\cdot 2^{n}}(x_2))\leq
 1/2^{n}\cdot 4\cdot \sum_{\ell=0}^{\infty}2^{-\ell}= 2^{-n+3},
 \end{align*}
 for all $k\geq n$.
 Now, given $y\in \mc T$, let $(x_n)_{n\in\N}$ be a sequence in $\mc T'$ with $d(x_n,y)\leq 2^{-2^{n}}$.
 Observe that $D^n(x_n,y)\leq 1/2^n\cdot \sum_{\ell=1}^{2^{n}}2^{-\ell}\leq 2^{-n}$ as well as
 $d(x_n,x_k)\leq 2^{-2^{n}}$ for $k\geq n$.
 Hence, $D^k(x_n,y)\leq D^k(x_n,x_k)+D^k(x_k,y)\leq 2^{-n+3}+2^{-k}$ for all $k\geq n$ so that
 $D_\Fol(x_n,y)\leq 2^{-n+3}$.
 This yields the $\Fol$-mean equicontinuity of $(\mc T,\sigma)$.
 Observe that $\mc T$ contains a dense set of points which are periodic with
 respect to $\sigma$ as well as a dense set of infinite (i.e., non-periodic)
 subshifts (in fact, regular Toeplitz subshifts).
\end{example}


\section{Mean equicontinuity and discrete  spectrum}
\label{sec:mean equicontinuity and pure point spectrum}

In this section we establish a relation between mean equicontinuity and discrete spectrum.
A dynamical system $(X,G)$ together with an invariant measure $\mu$
is said to have \emph{discrete spectrum} if $L_2 (X,\mu)$ can be
written as an orthogonal sum of finite dimensional, $G$-invariant subspaces $V_\alpha$, where
$\alpha$ runs through some index set, see \cite{MacKey1964} for further details.
As before, we  will denote by $(\mef, G)$ the maximal equicontinuous factor of $(X,G)$
and by $\pi:X\to\mef$ a corresponding factor map.

We will need the following well-known fact which follows from the general theory
of Ellis semigroups of equicontinuous systems (see, for example, \cite[pp.~52--53]{Auslander1988}):
if $(\T',G)$ is minimal and equicontinuous, then $\T'$ is homeomorphic to a
\emph{homogeneous space}, that is, there is a compact group $E(\T')$ and a closed
subgroup $F\leq E(\T')$ (in general not normal) such that
$\T'$ is homeomorphic to the set of left cosets $E(\T')/F$.
If $G$ is abelian, then $E(\T')$ is abelian and $\T'$ is homeomorphic to $E(\T')$.

\begin{theorem}\label{thm: unitary and pure spectrum}
    Suppose $(X,G)$ is minimal.
    Then, the following assertions are equivalent:
    \begin{itemize}
        \item[(i)] The system $(X,G)$ is mean equicontinuous.
        \item[(ii)]  $(X,G)$ is uniquely ergodic and, if $\mu$ denotes the unique
            invariant probability measure, then $L_2 (X,\mu)$ can be written as
            an orthogonal sum of finite dimensional, $G$-invariant subspaces
            $V_\alpha$, consisting of continuous functions ($\alpha$ runs through
            some index set $I$).
\end{itemize}
\end{theorem}
\begin{proof}
    (i)$\Rightarrow$(ii): Several results of the previous sections imply that every minimal mean
    equicontinuous system is uniquely ergodic.
    Let us hence denote by $\mu$ the unique $G$-invariant measure on $X$.
    Now, observe that if $L_2(\mef,\pi(\mu))$ can be decomposed
    as an orthogonal sum of finite dimensional $G$-invariant
    subspaces consisting of continuous functions, then this holds true for $L_2(X,\mu)$ as well.
    This follows from the unitarity of $U_\mu$ (defined as in \eqref{eq: defn unitary operator};
    see also Theorem \ref{thm:topo-isomorphy implies mean equicontinuity}
    and Proposition \ref{prop:topo-isomorphy-via-unitary-operators})
    and the fact that $U_\mu$ maps continuous functions to continuous functions (due
    to the continuity of $\pi$).

    Therefore, it suffices to find a corresponding decomposition
    of $L_2(\mef,\pi(\mu))$.
    If $\T$ is homeomorphic to the compact group $E(\T)$ from above, this decomposition is provided by the classical Peter-Weyl Theorem.
    In case that $\T$ is homeomorphic to a homogeneous space, the decomposition is obtained by a standard
    extension of the Peter-Weyl Theorem to homogeneous spaces.

    \smallskip

    (ii)$\Rightarrow$(i): For each $\alpha \in I$ we define the pseudometric
    $d_\alpha$ on $X$ via
    \[
        d_\alpha (x,y) \= \sup\{\abs{f(x) - f(y)} \,|\, f\in V_\alpha,\|f\|_\infty = 1\}.
    \]
    As $V_\alpha$ is finite dimensional and consists of continuous functions,
    each $d_\alpha$ is continuous.
    As $V_\alpha$ is $G$-invariant, each $d_\alpha$ is $G$-invariant.
    Further, observe that the separability of $L_2(X,\mu)$ implies that $I$ is countable.
    Thus, we may consider the pseudometric
    $D'=\sum_\alpha c_\alpha \cdot d_\alpha$, where $(c_\alpha)_{\alpha\in I}$ is
    some summable sequence of positive numbers.

    We can hence introduce an invariant and closed equivalence relation on $X$ by
    \[
        x\sim y :\Longleftrightarrow d_\alpha (x,y) = 0 \mbox{ for all }
        \alpha \in I\quad (\Longleftrightarrow D'(x,y)=0).
    \]
    Then $Y\=X /\!\!\sim$ is a compact space which we may consider equipped with the
    metric $D'$ in the obvious way.
    Further, $(Y,G)$ (where the action of $G$ on $Y$ is defined
    in the canonical way) is an isometric and hence equicontinuous
    factor, as $D'$ is $G$-invariant.
    Let $h : X\to Y$ be the factor map and note that
    $V : L^2 (Y,h(\mu))\to L^2 (X,\mu)$ with $Vf = f\circ h$ is unitary
    (as we only identify points which can not be
    distinguished by elements of the $V_\alpha$).
    Now, the application of Theorem \ref{thm:topo-isomorphy implies mean
    equicontinuity} yields (i).
\end{proof}

\begin{remark}\mbox{}
\begin{itemize}
    \item[(a)] The assumption of minimality of $(X,G)$ can be slightly weakened
        to  unique ergodicity or transitivity, where the latter implies the former due to Theorem~\ref{thm: decomposition theorem}.
        In fact, for (ii)$\Rightarrow$(i) we did not need the minimality of $(X,G)$.
        For (i)$\Rightarrow$(ii) note that $(\T,G)$ is still minimal (see
        Proposition~\ref{prop:hilfe}). 
    \item[(b)] The Peter-Weyl Theorem used in the proof of (i)$\Rightarrow$(ii)
        actually gives one more feature of the finite dimensional subspaces appearing in (ii).
        They can be assumed to be irreducible. Here, a $G$-invariant  subspace $V$
        of $L_2 (X,\mu)$ is called \emph{irreducible} if it can not be written as
        an orthogonal sum of two non-trivial $G$-invariant subspaces.
\end{itemize}
\end{remark}

In the case of abelian $G$, we obtain a somewhat stronger statement.
As this is of interest in various contexts, we include a discussion.

\begin{corollary}
    Let $G$ be abelian. Suppose $(X,G)$ is minimal.
    Then, the following assertions are equivalent:
    \begin{itemize}
        \item[(i)] The system $(X,G)$ is mean equicontinuous.
        \item[(ii)]  The system $(X,G)$ is uniquely ergodic and, if $\mu$ denotes
            the unique invariant probability measure, then $L^2 (X,\mu)$ has an
            orthonormal basis of continuous eigenfunctions.
\end{itemize}
\end{corollary}
\begin{proof} Clearly, condition (ii) of the present corollary  is stronger than condition
(ii) of the previous theorem. Thus, it suffices to show
(i)$\Rightarrow$(ii).
This can be seen  as in the proof of the
previous theorem after noting that $\T$ is homeomorphic to the compact group $E(\mef)$.
With this in mind, statement (ii) is a direct consequence of the duality theory for
compact abelian groups.
Alternatively, one may also argue that the
irreducible subspaces appearing in Theorem \ref{thm: unitary and
pure spectrum} (ii) must be  one-dimensional in the abelian case.
\end{proof}

\begin{remark}\label{rem:mean equicontinuity and diffraction measure}
The last three decades have seen tremendous interest in the field of
aperiodic order, also known as mathematical quasicrystals (see
\cite{BaakeGrimm2013,KellendonkLenzSavinien2015} for extensive
discussions).  The common way to model aperiodic order is via
dynamical systems over the group $\R^n$. In typical examples, these
systems will be uniquely ergodic and minimal. In any case, such a
system comes with a diffraction measure. As mentioned in the
introduction, a key effort is to show that the diffraction measure
is a pure point measure. This in turn has been proven to be
equivalent to discrete spectrum of the underlying dynamical system.
Hence, discrete spectrum is at the core of aperiodic order. In the
further analysis of the  diffraction measure, continuity of the
eigenfunctions turns out to play a role. Indeed, it is exactly under
this condition that a convincing positive answer to the so-called
Bombieri-Tayler Conjecture can be given \cite{Lenz2009} (see
\cite{Robinson1999} for related earlier results as well.) Given this
situation, the class of minimal uniquely ergodic systems with
discrete spectrum and continuous eigenfunctions (which is
characterized in the preceding corollary)  presents itself as a very
natural candidate for models of aperiodic order.
\end{remark}


\section{Non-abelian examples and conclusions}
\label{sec: examples and applications}

\subsection{Isometric subgroups of topological full groups}
\label{sec:isometric subgroups}

In this section, we provide means to construct new examples of (in particular, non-abelian)
mean equicontinuous group
actions by using suitable subgroups of the topological full group of known mean
equicontinuous systems.
Recall that the \emph{topological full group} $\tfg{(Z,G)}$ of a dynamical system $(Z,G)$
is the group of all homeomorphisms on $Z$ which
locally coincide with an element of $G$
equipped with the uniform topology.
For simplicity, let us restrict to systems where $Z$ is a Cantor space.
In this case, a homeomorphism $s:Z\to Z$ is an element of $\tfg{(Z,G)}$ if and only
if for every $z_0\in Z$ there is a clopen neighborhood $U$ of $z_0$ and an element
$g\in G$ such that $sz=gz$ for all $z\in U$.

We make use of the following structural result.
Recall that a group $G$ acts \emph{freely} on $Z$ if $gz=z$ for some
$z\in Z$ and $g\in G$ implies that $g$ is the identity.

\begin{theorem}[{\cite[Corollary 4.9]{CortezMedynets2016}}]\label{thm: tfg is amenable}
    Suppose $Z$ is a Cantor space and $G$ is countable.
    Further, assume $G$ acts minimally, equicontinuously and freely on $Z$.
    Then the topological full group $\tfg{(Z,G)}$ is amenable if and only if $G$ is
    amenable.
\end{theorem}

Suppose we are in the situation of the previous statement, in particular, $G$ acts
equicontinuously on $(Z,d)$.
Without loss of generality we may assume that $G$ acts isometrically with respect
to $d$ (see Section~\ref{sec:Notation}).
We define the \emph{isometric subgroup} $\tfg{(Z,G)}_I\leq \tfg{(Z,G)}$ to be that
subgroup which comprises all elements of $\tfg{(Z,G)}$ that act isometrically on
$Z$ with respect to $d$.
Clearly, $\tfg{(Z,G)}_I$ is a closed subgroup of $\tfg{(Z,G)}$ and hence amenable
due to Theorem \ref{thm: tfg is amenable}.

Now, if $(X,G)$ is an extension of $(Z,G)$ via the factor map $h:X\to Z$ with
$Z$ a Cantor space, then $\tfg{(Z,G)}_I$ acts naturally on $X$: given
$s\in\tfg{(Z,G)}_I$ with a (finite) clopen partition $\{Z_i\}$ of $Z$ and
elements $\{g_i\}\subseteq G$ such that $s|_{Z_i}=g_i|_{Z_i}$, then let $sx=g_i x$
whenever $x\in h^{-1}(Z_i)$.
Furthermore, we immediately see that $h$ still is a factor map from $(X,\tfg{(Z,G)}_I)$ to
$(Z,\tfg{(Z,G)}_I)$.

\begin{theorem}\label{thm: isometric subgroups of topological full group}
    Suppose $(X,G)$ is mean equicontinuous and uniquely ergodic with a maximal equicontinuous
    factor $(\mef,G)$ where $\mef$ is a Cantor space.
    If $G$ is countable and acts freely on $\mef$, then we have that $(X,\tfg{(\mef,G)}_I)$
    is mean equicontinuous with MEF $(\mef,\tfg{(\mef,G)}_I)$.
\end{theorem}
\begin{proof}
    Since $(X,G)$ is uniquely ergodic, Proposition~\ref{prop:hilfe} yields that $(\T,G)$ is minimal.
    Hence, $(\mef,G)$ verifies the assumptions of Theorem~\ref{thm: tfg is amenable}.
    Clearly, every $\tfg{(\mef,G)}_I$-invariant measure on $X$ (on $\mef$) necessarily
    is also a $G$-invariant measure on $X$ (on $\mef$).
    Therefore, $(X,\tfg{(\mef,G)}_I)$ is a topo-isomorphic extension of the equicontinuous
    system $(\mef,\tfg{(\mef,G)}_I)$.
    By Theorem~\ref{thm:topo-isomorphy implies mean equicontinuity}, the statement follows.
\end{proof}

Last, we present a straightforward instructive application of
Theorem \ref{thm: isometric subgroups of topological full group}.
Let us point out that all the considerations in the following example directly
generalize to higher-dimensional odometers and associated regular
Toeplitz configurations, see \cite{Cortez2006}.

\begin{example}
    We assume that the reader is familiar with the theory of odometers/adding
    machines, see for instance \cite{Kurka2003,Downarowicz2005} for further information.
    We consider the dyadic odometer $(\mathbf{2}^\N,\Z)$.
    That is, $\mathbf{2}^\N$ is the compact group obtained as the inverse limit
    \[
        \mathbf{2}^\N\=\varprojlim\limits_{\ell \in \N}\Z/2^\ell\Z,
    \]
    and $n\in \Z$ acts on $\theta\in \mathbf{2}^\N$ by $\theta\mapsto\theta+n$,
    where we consider $n$ as an element of $\mathbf{2}^\N$.

    Now, the isometric subgroup $\tfg{(\mathbf{2}^\N,\Z)}_I$ contains, among others,
    the element $s$ given by
    \[
    s\theta=s(\theta_0,\theta_1,\ldots)\=
    \begin{cases}
      \theta & \text{ if } \theta_0=0,\\
      \theta+2 & \text{ if } \theta_0=1.
    \end{cases}
    \]
    Obviously, $s(\theta+1)\neq 1+s\theta$.
    Hence, $\tfg{(\mathbf{2}^\N,\Z)}_I$ is a non-abelian amenable group which acts
    mean equicontinuously (according to Theorem~\ref{thm: isometric subgroups of topological full group})
    on, in particular, the shift orbit closure of any regular
    Toeplitz sequence whose MEF is given by $(\mathbf{2}^\N,\Z)$ (for concrete
    examples, see also \cite{Kurka2003,Downarowicz2005}).
    Obviously, these orbit closures are Cantor spaces as well.
    To obtain examples where the domain is not totally disconnected, we can
    consider Auslander systems, see \cite{HaddadJohnson1997}, which also
    have odometers as their MEF and are mean equicontinuous.
\end{example}

\subsection{Irregular extensions}\label{sec: irregular extensions}

Suppose $(X,G)$ is an extension of $(Y,G)$ via the factor map
$h:X\to Y$. We say $(X,G)$ is a \emph{regular} extension of $(Y,G)$
if for every $G$-invariant measure $\mu$ on $X$ we have that
$h(\mu)(\{y\in h(X)\: \#h^{-1}(y)>1\})=0$; otherwise we say $(X,G)$ is
an \emph{irregular} extension.
Given $y\in Y$, we refer to $h^{-1}(y)$ as its \emph{fiber}.

Note that a regular extension is automatically a
topo-isomorphic extension. Examples of regular extensions of
equicontinuous systems are Sturmian subshifts, regular Toeplitz
subshifts and the Denjoy systems described in Example \ref{ex:Denjoy
homeomorphisms}.
There are also irregular topo-isomorphic extensions of equicontinuous systems.
The Cantor substitution subshift in Example
\ref{ex:Cantor substitution} is a transitive irregular extension of
the trivial system.
Minimal examples can be found in \cite{DownarowiczKasjan2015,DownarowiczGlasner2016},
where \cite[Example~5.1]{DownarowiczKasjan2015} has almost surely (with respect
to the unique invariant measure of its MEF) countable
fibers but still a residual set of points whose fibers are singletons.
In contrast, in the examples constructed in \cite[Section 3]{DownarowiczGlasner2016},
every fiber is uncountable.
Indeed, in this subsection, we will show that almost every fiber of an irregular
extension must be at least countable.
For the convenience of the reader, we provide a proof of the
next statement.

\begin{lemma}\label{lem: almost surely same cardinality in fibres}
    Let $(X,G)$ be an extension of $(Y,G)$ via the factor map $h:X\to Y$ and let
    $\mu$ be an ergodic $G$-invariant measure on $Y$.
    Suppose $h^{-1}(y)$ is finite for $\mu$-almost every $y\in Y$.
    Then there is $n_0\in\N$ such that $\mu$-almost everywhere we have $\#h^{-1}=n_0$.
\end{lemma}
\begin{proof}
    Observe that $h$ gives rise to an upper semi-continuous and hence Borel measurable
    map $\gamma$ from $Y$ to the space of compact subsets of $X$ (endowed with the
    Hausdorff metric), defined by $\gamma(y)=\pi^{-1}(y)$ for each $y\in Y$.
    By Lusin's Theorem, there is a compact set $K'\subseteq Y$ with $\mu(K')>0$
    such that $\gamma\left.\right|_{K'}$ is continuous.
    Set $Y'\=\{y\in Y\:\# h^{-1}(y)<\infty\}$.
    By the assumptions, $\mu(Y')=1$.
    Since $\mu$ is an inner regular measure, we may assume w.l.o.g.\ that $K'\ssq Y'$.
    Let $K\subseteq K'$ be the support of the measure $\mu|_{K'}$.
    Clearly, $\mu(K)>0$.

    Pick some $y_0\in K$ and set $n_0\=\#h^{-1}(y_0)$.
    By continuity of $\gamma$ on $K$, there is $\delta>0$ such that for all
    $y\in B_\delta(y_0)\cap K$ we have
    \[
        d(\gamma(y),\gamma(y_0))<\frac{1}{2}\cdot\min_{x_1\neq x_2 \in h^{-1}(y_0)}d(x_1,x_2)
    \]
    and hence $\#h^{-1}(y)\geq n_0$.
    Since $G$ acts on $X$ by homeomorphisms, we actually have $\#h^{-1}(y)\geq n_0$
    for each $y$ in the invariant set $A=\bigcup_{g\in G} g\, (B_\delta(y_0)\cap K)$.
    By definition of $K$, $\mu(B_\delta(y_0)\cap K)>0$ so that $A$ is of full
    measure, since $\mu$ is ergodic.

    Set $\mc N(K)\=\{n\in \N\: \text{there is } y\in K \text{ such that } \#h^{-1}(y)=n\}$.
    If we can show that $\mc N(K)$ is bounded, the above proves the statement.
    Assume for a contradiction that $\mc N(K)$ is unbounded.
    The above shows: for all $n\in \mc N(K)$, we have $\mu(\{y\in Y\:\#h^{-1}(y)\geq n\})=1$ or, in other words,
    $\mu(\{y\in Y\:\#h^{-1}(y)< n\})=0$.
    Hence, $\mu(Y')=\mu(\bigcup_{n\in \mc N(K)}\{y\in Y\:\#h^{-1}(y)< n\})=0$ which is an obvious contradiction.
\end{proof}

\begin{theorem}
    Let $(X,G)$ be an irregular extension of $(Y,G)$ via the factor map $h:X\to Y$,
    that is, $h(\mu)(\{y\in h(X)\: \#h^{-1}(y)>1\})>0$ for some $G$-invariant measure $\mu$ on $X$.
    If $h(\mu)$-almost all fibers of $h$ are finite, then $h$ is not a topo-isomorphy.
\end{theorem}
\begin{proof}
    We may assume without loss of generality that $\mu$ is ergodic, because of
    Lemma \ref{lem: ergodic representation}.
    For a contradiction, assume that $(X,G)$ is a topo-isomorphic extension
    of $(Y,G)$ via $h$.

    By definition of topo-isomorphy, there is a Borel measurable map $\gamma:Y \to X$
    such that for all $g\in G$ we have $\gamma(gy)=g\gamma(y)$ for $h(\mu)$-almost all $y$
    and $\int\!\varphi\, d\mu=\int\!\varphi\circ\gamma \,dh(\mu)$ for every $\varphi\in\mc C(X)$.
    By Lemma~\ref{lem: almost surely same cardinality in fibres}, we further have that $h(\mu)$-almost
    all fibers are of equal cardinality $n_0>1$.
    By the Riesz Representation Theorem, we can define a probability measure $\nu$ on $X$ by
    \[
        \varphi\longmapsto\int\limits_{Y}\!\frac{1}{n_0-1}\cdot
            \sum\limits_{\mathclap{\substack{x\in h^{-1}(y)\\x\neq\gamma(y)}}}\varphi(x)\,dh(\mu)(y)
            \qquad(\varphi\in \mc C(X)).
    \]
    Observe that $\nu$ is $G$-invariant and $\nu\neq\mu$.
    Moreover, $h(\nu)=h(\mu)$ and this implies $\nu$ is ergodic since $h$ is assumed to be a topo-isomorphy.
    However, this yields a contradiction, according to
    Proposition \ref{prop: unitarity implies distant ergodic components}.
\end{proof}

We  immediately obtain the next two statements, using
Corollary~\ref{cor: 1} for the second one.

\begin{corollary}
    Assume $(X,G)$ has a unique $G$-invariant measure $\mu$ and is an irregular
    topo-isomorphic extension of a mean equicontinuous system $(Y,G)$ via a
    factor map $h:X\to Y$.
    Then for $h(\mu)$-almost every $y\in Y$ we have that $h^{-1}(y)$ is infinite.
\end{corollary}

\begin{corollary}
    Suppose that $(X,G)$ is an irregular extension of $(Y,G)$ via the factor map
    $h:X\to Y$ and suppose the MEF of $(X,G)$ and $(Y,G)$ coincide.
    If $(Y,G)$ is mean equicontinuous and the fibers of $h$ are finite,
    then $(X,G)$ can not be mean equicontinuous.
\end{corollary}

An example fitting into the setting of the second corollary is the
Thue-Morse subshift which is a 2-1 extension of a regular Toeplitz
subshift with the same maximal equicontinuous  factor (see, for instance, \cite{BaakeGrimm2013}
for more information).
In particular, we get that the Thue-Morse system is not mean equicontinuous.

\subsection{Maximally almost periodic groups}\label{MAP-groups}

In this last section we show that if a group $G$ acts minimally, mean equicontinuously
and effectively on a compact metric space $(X,d)$, then it is necessarily maximally
almost periodic.

Recall that $G$ acts \emph{effectively} on $X$ if for each
$g\in G$, there is $x\in X$ with $gx\neq x$.
Recall further that a topological group $G$ is \emph{maximally almost periodic (MAP)}
if $G$ admits a continuous and injective homomorphism into a compact Hausdorff group,
see for instance \cite{Neumann1934}.
Note that a locally compact MAP group is necessarily unimodular \cite{LeptinRobertson1968}.
We will make use of the following characterization of maximal almost periodicity
\cite{Huang1979}: a topological group $G$ is MAP if and
only if $G$ admits an equicontinuous and effective action on a compact Hausdorff space.

\begin{theorem}\label{thm: effecivenes is inherited by the MEF}
    Let $(X,G)$ be a dynamical system and denote by $(\mef,G)$ its maximal
    equicontinuous factor.
    If $(X,G)$ is mean equicontinuous, allows for an invariant measure of full
    support and $G$ acts effectively on $X$, then $G$ also acts effectively on $\mef$.
    In particular, this implies that $G$ is maximal almost periodic and unimodular.
\end{theorem}
\begin{proof}
    As before, we denote by $\pi$ a factor map from $X$ to $\mef$.
    Let $\mu$ be an invariant measure with full support.
    Since $\pi$ is a topo-isomorphy, there
    are subsets $M\ssq X$ and $N\ssq\mef$ of full $\mu$- and $\pi(\mu)$-measure, respectively, such that
    the restriction of $\pi$ to $M$ is a bijection from $M$ onto $\pi(M)=N$.

    Now, assume there is $g\in G$ with $gy=y$ for all $y\in\mef$.
    Observe that such $g$ has to verify $gx=x$ for $\mu$-almost all $x\in M$, since $\pi$ restricted
    to $M$ is injective and since --by the invariance of $\mu$-- almost every point of $M$ is mapped into $M$ under the action of $g$.
    As $\mu$ is of full support, every full-measure set is dense in $X$.
    Thus, the continuity of $g$ implies $gx=x$ for all $x\in X$.
    As $G$ acts effectively on $X$, this gives $g=e$.
\end{proof}

Recall that for a minimal dynamical system all measures have full support.

\begin{corollary}
    If $G$ acts minimally, mean equicontinuously and effectively on $X$, then
    $G$ is maximal almost periodic and unimodular.
\end{corollary}

We would like to close with a partial answer to the following question \cite[Question~8.1]{GlasnerMegrelishvili2018}:
which discrete countable groups $G$ have effective tame minimal actions?
Here, the term \emph{tame} refers to a certain low dynamical complexity of a dynamical system (see, e.g., \cite{Glasner2018}).
Now, according to \cite[Corollary~5.4 (2)]{Glasner2018}, if $(X,G)$ is tame and $G$ amenable, then $(X,G)$ is a topo-isomorphic
extension of its MEF and hence mean equicontinuous, due to Theorem~\ref{thm:topo-isomorphy implies mean equicontinuity} (for
$\Z$-actions, see also \cite[Corollary~5.10]{Glasner2018}).
Thus, from Theorem~\ref{thm: effecivenes is inherited by the MEF}
we obtain that among the amenable, discrete countable groups exactly the maximally almost periodic ones allow
for an effective tame minimal action.

\small
\bibliography{lit}
\bibliographystyle{alpha}

\end{document}